\documentclass{article}
\usepackage{amsmath,amssymb,amsthm,graphicx,epsfig,float,url}
\usepackage[colorlinks=true,citecolor={Plum},linkcolor={Periwinkle}]{hyperref}
\usepackage{pdfsync}
\usepackage[usenames, dvipsnames]{xcolor}
\usepackage{tikz}
\usepackage{subfig}
\usepackage{bbm}
\usepackage{stmaryrd}
\usepackage{amssymb}
\usepackage{mathrsfs}  
\usepackage{caption}
\usepackage{float}
\usepackage{esint}

\usepackage{dsfont}
\usepackage[makeroom]{cancel}
\usetikzlibrary{patterns}
\newcommand{\R}{\textnormal{I\kern-0.21emR}}
\newcommand{\N}{\textnormal{I\kern-0.21emN}}

\newcommand{\A}{\mathcal{A}}

\renewcommand{\geq}{\geqslant}
\renewcommand{\leq}{\leqslant}

\def\B{{\mathbb B}}
\def\A{{\mathbb A}}
\def\e{{\varepsilon}}

\allowdisplaybreaks

\def\YYint#1#2#3{{\setbox0=\hbox{$#1{#2#3}{\iint}$}
    \vcenter{\hbox{$#2#3$}}\kern-.51\wd0}}
 
\usepackage[dvipsnames]{xcolor}

\newtheorem*{theorem*}{Theorem}

\newtheorem{theorem}{Theorem}

\newtheorem{material}{material}
\newtheorem{proposition}[material]{Proposition}

\newtheorem{definition}[material]{Definition}
\newtheorem{lemma}[material]{Lemma}
\newtheorem{claim}[material]{Claim}

\newtheorem{remark}[material]{Remark}

\def\O{{\Omega}}
\def\n{{\nabla}}
\def\p{{\varphi}}
\def\hut{{\hat u_\tau'}}
\def\OT{{(0;T)\times \O}}
\def\op{{\overline \p}}

\usepackage{xargs}
 \usepackage[colorinlistoftodos,textsize=small]{todonotes}
 \newcommandx{\christian}[2][1=]{\todo[linecolor=red,backgroundcolor=red!25,bordercolor=red,#1]{#2}}
 \newcommandx{\laura}[2][1=]{\todo[linecolor=blue,backgroundcolor=blue!25,bordercolor=blue,#1]{#2}}
 \newcommandx{\info}[2][1=]{\todo[linecolor=green,backgroundcolor=green!25,bordercolor=green,#1]{#2}}
 \newcommandx{\improvement}[2][1=]{\todo[linecolor=yellow,backgroundcolor=yellow!25,bordercolor=yellow,#1]{#2}}
 
  \newcommandx{\biblio}[2][1=]{\todo[linecolor=blue,backgroundcolor=magenta!25,bordercolor=blue,#1]{#2}}

 \numberwithin{equation}{section}

\begin{document}
\title{Quantitative estimates for parabolic optimal control problems under $L^\infty$ and $L^1$ constraints in the ball:\\Quantifying parabolic isoperimetric inequalities}


\author{Idriss Mazari\footnote{Technische Universit\"{a}t Wien, Institute of Analysis and Scientific Computing, 8-10 Wiedner Haupstrasse, 1040 Wien (\texttt{idriss.mazari@tuwien.ac.at})}}
\date{\today}

\maketitle

\begin{abstract} In this article, we present two different approaches for obtaining quantitative inequalities in the context of parabolic optimal control problems. Our model consists of a linearly controlled heat equation with Dirichlet boundary condition $(u_f)_t-\Delta u_f=f$, $f$ being the control. We seek to maximise the functional $\mathcal J_T(f):=\frac12\iint_\OT u_f^2$ or, for some $\e>0$\,, $\mathcal J_T^\e(f):=\frac12\iint_\OT u_f^2+\e\int_\O u_f^2(T,\cdot)$ and to obtain quantitative estimates for these maximisation problems. We offer two approaches in the case where the domain $\O$ is a ball. In that case, if $f$ satisfies $L^1$ and $L^\infty$ constraints and does not depend on time, we propose a shape derivative approach that shows that, for any competitor $f=f(x)$ satisfying the same constraints, we have $\mathcal J_T(f^*)-\mathcal J_T(f)\gtrsim \Vert f-f^*\Vert_{L^1(\O)}^2$, $f^*$ being the maximiser. Through our proof of this time-independent case, we also show how to obtain coercivity norms for shape hessians in such parabolic optimisation problems. We also consider the case where $f=f(t,x)$ satisfies a global $L^\infty$ constraint and, for every $t\in (0;T)$, an $L^1$ constraint. In this case, assuming $\e>0$, we prove an estimate of the form $\mathcal J_T^\e(f^*)-\mathcal J_T^\e(f)\gtrsim\int_0^T a_\e(t) \Vert f(t,\cdot)-f^*(t,\cdot)\Vert_{L^1(\O)}^2$ where $a_\e(t)>0$ for any $t\in (0;T)$. The proof of this result relies on a uniform bathtub principle.
\end{abstract}


\medskip


\paragraph{Acknowledgment.}I Mazari was supported by the Austrian Science Fund (FWF) through the grant I4052-N32 . I. Mazari was partially supported by the French ANR Project ANR-18-CE40-0013 - SHAPO on Shape Optimization and by the Project "Analysis and simulation of optimal shapes - application to lifesciences" of the Paris City Hall.


\section{Introduction}
This Introduction is structured as follows: 
 in Subsection \ref{Se:Scope}, we present the scope of our article; in Subsection \ref{Se:Goal}, we give an informal statement of our results while in Subsection \ref{Se:Biblio} we give several bibliographical references on qualitative properties for optimal control problems, shape derivatives for parabolic problems and quantitative inequalities.
 In Subsection \ref{Se:Schwarz}, we give basic information regarding the Schwarz rearrangement, which will be a key tool in our analysis, and we give bibliographical references for parabolic isoperimetric inequalities.
 In Subsection \ref{Se:Results}, we state our main results, Theorems \ref{Th:Td} and \ref{Th:Ti} (Theorem \ref{Th:Uniqueness} deals with the uniqueness of solutions to our optimal control problem and is also stated in this Section).
 In Subsection \ref{Se:Plan} we present the plan of our paper and, finally, in Subsection \ref{Su:Notation}, we gather the notations we will use throughout the paper.

\subsection{Scope of the article}\label{Se:Scope}

\subsubsection{Goal of this article: informal statement of the problems and of the results}\label{Se:Goal}

In this article, our goal is to present two different approaches for obtaining \textit{quantitative inequalities for optimal control problems}, which will also be dubbed \emph{quantitative isoperimetric parabolic inequalities}. Before explaining how this fits in the growing field of qualitative questions in optimal control theory, let us vaguely state the type of results we wish to establish, and sketch the two approaches that will be put forth. By quantitative inequalities, we mean the following: we consider a controlled parabolic partial differential equation assuming the general form
\begin{equation}\label{Eq:Intro}
u_t-\mathcal Lu=f \text{ in }(0;T)\times \O,
\end{equation}
$\mathcal L$ being an elliptic operator; this equation is 
supplemented with some initial condition and some boundary conditions. In this setting, $f$ is the control and depends \textit{a priori} both on time and space. It is assumed to satisfy some constraints, which will be taken into account by assuming that $f\in \mathscr X$, where $\mathscr X$ is some subset of a function space. The cost to be optimised is some functional $\mathcal J_T:\mathscr X \ni f\mapsto \mathcal J_T(f)$. The control problems reads
\begin{equation}\label{Eq:PvIntro}
\fbox{$\displaystyle \max_{f\in \mathscr X}\mathcal J_T (f).$}\end{equation} 

The quantitative inequality we aim at can take two different forms:
\begin{itemize}
\item \textbf{For time independent controls.} In the context where all controls $f\in \mathscr X$ write $f=f(x)$, and if the solution of \eqref{Eq:PvIntro}  is some $\overline f$ (assumed to be unique for simplicity), the goal is to establish the following kind of estimate
\begin{equation}\label{Eq:IntroIndep}\fbox{$\displaystyle \forall f \in \mathscr X\,, 
\mathcal J_T(f)-\mathcal J_T(\overline f)\leq - C(T)\Vert f-\overline f\Vert_{L^1(\O)}^2$}\end{equation} for some constant $C(T)>0$.
The right-hand side quantity is natural in the context of quantitative inequalities for shape optimisation problems \cite{FuscoMaggiPratelli} and optimal control problems \cite{MazariQuantitative}, and is akin to the Fraenkel asymmetry. We refer to Subsection \ref{Se:Biblio}.
\item \textbf{For time-dependent controls.} In the context where the controls are time dependent i.e. $f=f(t,x)$ and when the solution of \eqref{Eq:PvIntro} is some $f^*$, the goal is to establish something of the form
\begin{equation}\label{Eq:IntroDep}\fbox{$\displaystyle 
\forall f\in \mathscr X\,, 
\mathcal J_T(f)-\mathcal J_T(f^*)\leq- \int_0^T \omega(s)\Vert f(s,\cdot)-f^*(s,\cdot)\Vert_{L^1(\O)}^2$}
\end{equation}
for a function $\omega:[0;T]\rightarrow \R_+$ such that for any $s\in (0;T)$, $\omega(s)>0$. As will be explained more in detail in Subsection \ref{Se:Biblio} and commented upon in the Conclusion, see Section \ref{Cl:Time}, this is a stronger norm than the usual one.

\end{itemize}
To the best of our knowledge, neither type of quantitative estimates have been derived despite their natural interest.

Obviously, one can not expect to prove \eqref{Eq:IntroIndep} or \eqref{Eq:IntroDep} for all optimal control problems. What we propose here is to \textit{establish both these inequalities for a linearly controlled heat equation in the ball under $L^1$ and $L^\infty$ constraints.}  The main equation under consideration is set in $\O=\mathbb B(0;R)$ and writes
\begin{equation}
\label{Eq:MainIntro}
\begin{cases}
\frac{\partial u_f}{\partial t}-\Delta u_f=f\text{ in }\OT\,, 
\\ u_f(t=0)=u^0\geq 0\text{ in }\O\,, 
\\u_f(t,\cdot)=0\text{ on }\partial \O.\end{cases}\end{equation} We will also assume that the initial condition $u^0\in \mathscr C^{2}(\O)\cap W^{1,2}_0(\O)$, $u^0\geq 0$, which is fixed, is radially symmetric and non-increasing. In the time-independent case (when $f=f(x)$), the functional we seek to maximise is defined by 
\begin{equation}
\mathcal J_T(f):=\frac12\iint_\OT u_f^2(t,x)dxdt.
\end{equation}
In the time dependent case (when $f=f(t,x)$), the functional we seek to maximise is
\begin{equation}
\mathcal J^\e_T(f):=\frac12\iint_\OT u_f^2(t,x)dxdt+\frac\e2 \int_\O u_f^2(T,x)dx
\end{equation}
for some $\e>0$. The main reason behind supplementing the functional with a final time term is to ensure the non-degeneracy of the switch function associated to the optimisation problem.

As a final comment, let us remark that the constant $C(T)$ appearing in \eqref{Eq:IntroIndep} and the weight $\omega$ are constructed in a non-explicit way.

\begin{remark}
Although we prove our results for maximisation of functionals, we believe the same strategies work for the minimisation of the functional. For both problems, both inequalities may have interesting consequences for inverse problems.\end{remark}

\begin{remark}
Obviously, if the optimal control $f^*$ for the time-dependent case does not depend on time, which will be the case here, \eqref{Eq:IntroDep} implies \eqref{Eq:IntroIndep} with $C(T)=\int_0^T \omega(s)ds.$ However, the reason why we present two proofs is the possibility of generalising the methods used to prove \eqref{Eq:IntroIndep}. In the conclusion, we explain why we believe this inequality can be extended to other types of control problems, such as bilinear control problems, or how, in general domains, technical assumptions on second order shape derivatives may enable one to derive it.  In short, the proof of \eqref{Eq:IntroIndep} relies on two properties of the control problem: the first one is shape derivatives, for which the trickiest part is to prove coercivity of second order derivatives in general domains; the second one is the convexity of the problem, which is something extremely general.

The proof of \eqref{Eq:IntroDep} is specific to the case of the ball and it is unclear whether or not it may be adapted to other domains. Indeed, the parabolic isoperimetric inequalities used in its proof \cite{Alvino1990,Bandle,RakotosonMossino} may not hold in general domains, in the sense that explicit characterisation of maximisers may not be attainable.\end{remark}
\subsection{Statement of the main results}\label{Se:Results}
Let $\O=\mathbb B(0;R)$ be a centred ball in dimension $n$. We assume that we are given an initial condition satisfying 

\begin{equation}\label{Eq:A0}
u^0\geq 0\,, u^0\in \mathscr C^2(\O)\cap W^{1,2}_0(\O)\,, u^0\text{ is radially symmetric and non-increasing}.\end{equation} For a function $f\in L^2((0;T)\times \O)$ we consider the solution $u_f$ of
\begin{equation}
\label{Eq:Main}
\begin{cases}
\frac{\partial u_f}{\partial t}-\Delta u_f=f\text{ in }(0;T)\times  \O\,, 
\\ u_f(t=0)=u^0\,, 
\\u_f(t,\cdot)=0 \text{ in }(0;T)\times\partial  \O.\end{cases}\end{equation}
The functional we wish to optimise in the time-independent case is 
\begin{equation}\tag{$J$}\label{Eq:Fonc}
\mathcal J_T(f):=\frac12\iint_\OT u^2_f(T,\cdot).\end{equation} For a given $L^1$ constraint $V_0\in (0;\operatorname{Vol}(\O))$, the sets of admissible controls are 
\begin{equation}\label{Eq:AdmTI}\tag{$\bold{\overline{Adm}}$}
\overline{ \mathcal M}(\O):=\left\{f\in L^\infty(\O)\,, 0\leq f\leq 1\,, \int_\O f=V_0\right\}\end{equation} and

\begin{equation}\label{Eq:AdmTD}\tag{$\bold{Adm}_T$}
\mathcal M_T(\O):=\left\{f \in L^\infty\left((0;T)\times \O\right) \text{ for a.e. $t\in (0;T)$, } f(t,\cdot) \in \overline{\mathcal M}(\O).\right\}\end{equation}
The first problem we address is 
\begin{equation}\label{Eq:PvTi}\tag{$\bold I_1$}
\fbox{$\displaystyle \max_{f \in \overline{\mathcal M}(\O)}\mathcal J_T(f).$}\end{equation}
The second problem is set in the time-dependent case and the functional we seek to optimise is, for some $\e> 0$, 
\begin{equation}\mathcal J^\e_T(f):=\frac12\iint_\OT u^2_f(T,\cdot)+\frac\e2\int_\O u_f^2(T,\cdot).\end{equation}
The second problem we address is
\begin{equation}\label{Eq:PvTd}\tag{$\bold I_2^\e$}
\fbox{$\displaystyle \max_{f \in {\mathcal M_T}(\O)}\mathcal J_T^\e(f).$}\end{equation}
Finally, we fix throughout the paper the notation 
\begin{equation}
f^*:=\mathds 1_{\B^*}
\end{equation}
 where $\mathbb B^*$ is the unique centred ball of volume $V_0$, and define $u^*$ as the solution of \eqref{Eq:Main} associated with $f\equiv f^*$.
 \begin{remark}
The existence of solutions of \eqref{Eq:PvTi} and \eqref{Eq:PvTd} are easy consequence of the direct methods of the calculus of variations.\end{remark}
As an easy corollary of \cite[Theorem 3]{Alvino1990}, $f^*$ is a maximiser of both \eqref{Eq:PvTi} and \eqref{Eq:PvTd}. To prove our results, the first step is the following Theorem:

\begin{theorem} \label{Th:Uniqueness}$f^*$ is the unique solution of \eqref{Eq:PvTi} and of \eqref{Eq:PvTd}.\end{theorem}

It is not the main goal of this paper, but the result is in itself interesting and relies on topological properties of some classes of functions defined \emph{via} rearrangements. We refer to Section \ref{Se:Uniq}  for the proof.

Let us now pass to the two main results of this article. We choose to state first the time-dependent case, as the result holds without any restriction on the dimension. In this case, we need to take $\e>0$. The reason behind this is technical, and amounts, to put it shortly, to forcing the switch function of the problem to be non-degenerate. We comment on this in the Conclusion.
 \def\S{{\mathbb S}}
 \begin{theorem}\label{Th:Td}

If $\e>0$, there exists a constant $C(\e,T)>0$ such that 

\begin{equation}\label{Eq:DpValTd}
\forall f \in \mathcal M_T(\O)\,, \mathcal J_T^\e(f)-\mathcal J_T^\e(f^*)\leq- C(\e,T)\int_0^T\left(-\frac{\partial p_\e^*}{\partial \nu}\right)_{\partial \B^*}\Vert f(t,\cdot)-f^*\Vert_{L^1(\O)}^2.\end{equation}
In the estimate above $p_\e^*$ solves 
\begin{equation}
\begin{cases}
\frac{\partial p_\e^*}{\partial t}+\Delta p_\e^*=-u^*\text{ in }\OT\,, 
\\ p_\e^*(T,\cdot)=\e u^*(T,\cdot)\,, 
\\ p_\e^*(t,\cdot)=0\text{ on }(0;T)\times \partial \O
\end{cases}
\end{equation}
and is the switch function of \eqref{Eq:PvTd}.

 \end{theorem}
 \begin{remark}
 Actually, when $\e>0$, we can even prove that there exists a constant $A(\e,T)>0$ such that
 \begin{equation}
 \forall t \in (0;T)\,, \left(-\frac{\partial p_\e^*}{\partial \nu}\right)_{\partial \S^*}\leq -A(\e,T).
 \end{equation}
 We however choose to keep the partial derivative of the switch function as it seems to us to be a more precise result.
 \end{remark}
 We then pass to the time-independent case, where the main innovation will be the use of shape derivatives. We include this result not only for the sake of completeness but also because this method seems, at this stage, generalisable to other domains, while it may not be the case for Theorem \ref{Th:Td}.
 
 \begin{theorem}\label{Th:Ti} Assume $n=2$.
 
 For any $T>0$ there exists a constant $C(T)>0$ such that 
 \begin{equation}\label{Eq:DpValTi}
 \forall f \in \overline{\mathcal M}(\O)\,, \mathcal J_T(f)-\mathcal J_T(f^*)\leq -C(T)\Vert f-f^*\Vert_{L^1(\O)}^2.
 \end{equation}
 \end{theorem}

\subsubsection{Bibliographical references}\label{Se:Biblio}

Let us now present the frameworks into which we think our present work fits.

\paragraph{Qualitative questions in optimal control problems}
The question of qualitative properties of optimal control problems has recently drawn a lot of attention. Indeed, in many situations, explicit computation of the optimal control is nearly impossible, and a line of research has emerged that deals with the question of knowing what optimal controls nearly look like, or whether or not these controls are (un)stable in a sense that has to be specified. Among all these qualitative queries, one may single out the following:
\begin{itemize}
\item \textbf{Insensitising controls.} The question of insensitising controls is a very natural one, and is a possible solution to the following question: given that it is often the case that one can not practically realise the exact control strategy and that some imperfections may arise, how can a \textit{robust} control strategy be constructed? In that context, the goal is to find an \textit{insensitising} optimal control. This question has been studied, for instance, in \cite{Alabau2,LissyPrivat} and, more recently, in \cite{LissyPrivatErvedoza}.
\item \textbf{The turnpike property.} The turnpike property states that, when dealing with time evolving optimal control problems, it is sometimes possible to actually find a \textit{nearly static} optimal strategy or, in other words, that the optimal control remains \textit{close}, in some sense that has to be quantified, to the solution of a stationary optimal control problem. First motivated by applications in economics \cite{Dorfman}, this field has been rapidly growing over the last decade and has found applications in many contexts (e.g. control of non-linear differential equations, control of the wave equation, control of semilinear heat equations, machine learning) \cite{2006,Cass1966,borjan,LanceTrelatZuazua,PighinSakamoto,TrelatZhang,TrelatZhangZuazua,Zuazua2017}. It has recently been derived for bilinear optimal control problems using quantitative inequalities for stationary optimisation problems \cite{MRB2020}.
\end{itemize} 

\paragraph{Shape derivatives for time-evolving problems}
Our work presents what is to the best of our knowledge the first detailed analysis of a second order shape derivative for a time-evolving optimal control problems (in the sense that a coercivity norm for the second order shape derivative is obtained), albeit it deals with shape derivatives with respect to a subdomain. Although the literature devoted to time-evolving optimal control problems is scarce, we would like to point to \cite{Moubachir2006} where a speed-method approach is presented, and to the recent preprint \cite{Harbrecht} where shape derivatives (with respect to the underlying domain $\O$) are computed and used to obtain numerical simulations of a shape optimisation problem.

\paragraph{Quantitative inequalities} The study of quantitative inequalities in shape optimisation problems is an enormous field. To mention a few works, we point to the seminal \cite{FuscoMaggiPratelli} for the quantitative isoperimetric inequality, and to \cite{BDPV} for quantitative spectral inequalities. Regarding quantitative inequalities for (stationary) control problems we refer to \cite{BrascoButtazzo} for a quantitative inequality for the natural Dirichlet energy, to \cite{CarlenLieb} for a quantitative spectral inequality (with respect to the potential) in $\R^n$ (both these works are done under $L^p$ constraints), to \cite{MazariQuantitative} for a quantitative spectral inequality in the ball under $L^1$ and $L^\infty$ constraints and to \cite{MRB2020} for a generalisation of this inequality to other domains, and for an application to the turnpike property. 

Let us comment on the type of estimates usually obtained: given a functional $\mathcal F:\O\mapsto \mathcal F(\O)$, a typical problem reads 
\begin{equation}\inf_{\O\,, \operatorname{Vol}(\O)=\overline V}\mathcal F(\O).\end{equation} Let us assume that, up to a translation, the unique minimiser of this functional is a ball $\overline \B$ of volume $\overline V$ (this is the case when $\mathcal F(\O)=\operatorname{Per}(\O)$), then the inequality obtained in \cite{FuscoMaggiPratelli} reads: there exists a constant $C>0$ such that, defining the Fraenkel asymmetry of $\O$ as 
\begin{equation}
\mathcal A(\O)=\inf_{x\in \R^n}\operatorname{Vol}\left( (x+\overline \B)\Delta \O\right)\end{equation} there holds 
\begin{equation}
\mathcal F(\O)-\mathcal F(\overline \B)\geq C\mathcal A(\O)^2.\end{equation}
In the case of estimate \eqref{Eq:IntroIndep}, the coercivity obtained is akin to this measure of asymmetry if the maximiser $\overline f$ writes $\overline f=\mathds 1_{\overline E}$: by defining $\mathcal J_T(E):= \mathcal J_T(\mathds 1_E)$ with a slight abuse of notation,  if we choose a competitor $f$ of the form $f=\mathds 1_E$ then estimate \eqref{Eq:DpValTi} rewrites
\begin{equation}
\mathcal J_T(\overline E)-\mathcal J_T(E)\geq C(T)\operatorname{Vol}\left(\overline E\Delta E\right)^2.\end{equation}

On the other hand, \eqref{Eq:DpValTd} may seem more surprising. If, indeed, we assume that $f^*(t,x)=\mathds 1_{E^*(t)}(x)$ and if the competitor $f$ is chosen to assume the form $f(t,x)=\mathds 1_{E(t)}(x)$ for some subset $E$ of $\O$ then, seeing $\overline E:=\cup_{t\in (0;T)}\{t\}\times E(t)$ and $\overline E^*:=\cup_{t\in (0;T)}\{t\}\times E^*(t)$ as subsets of the cylindrical domain $\OT$, the "natural quantity" that one should obtain should be the squared asymmetry of $\overline E$ with respect to $\overline E^*$, that is, $\operatorname{Vol}(\overline E\Delta \overline E^*)^2$. The Jensen inequality enables to recover this discrepancy. This stronger norm may be a consequence of having chosen a volume constraint for every $t$. It is unclear at this stage whether or not replacing the constraint
\begin{equation}\text{ for a.e. $t\in (0;T)$,} \int_\O f(t,\cdot)=V_0\end{equation} with a global constraint
\begin{equation}\iint_{\OT}f=V_0\end{equation} would yield the coercivity norm
\begin{equation}\left(\iint_\OT\left|f^*-f\right|\right)^2.\end{equation} We refer to the Conclusion, Section \ref{Cl:Time}.

\subsection{Schwarz's rearrangement and isoperimetric inequalities for parabolic equations}\label{Se:Schwarz}
In order to be able to comment on our results and methods of proof, we need to give the basic definition underlying most of our methods, that of Schwarz's rearrangement. The three books we refer to for a comprehensive introduction to rearrangements are \cite{Kawohl,Kesavan,Rakotoson}. Here, since we are already working in a ball $\O=\mathbb B(0;R)$, we only give the definitions for functions defined on the ball.

\begin{definition}\label{De:Schwarz}
For a function $\p\in L^2(\O)\,, \p\geq 0$, its Schwarz rearrangement is the unique radially symmetric non-increasing function $\p^\#:\O\to \R$ such that
\begin{equation}\forall t \in \R_+\,, \operatorname{Vol}\left(\{\p> t \}\right)=\operatorname{Vol}\left(\{\p^\#> t\}\right).\end{equation} We define its one-dimensional counter part $\p^\dagger:[0;R]\to \R$ as
\begin{equation}\p^\dagger(|x|):=\p^\#(x).\end{equation}
\end{definition} The first property is that the Schwarz rearrangement preserves all the Lebesgue norms:
\begin{equation}\label{Eq:Equimeasurability}
\forall p\in (1;+\infty)\,, \forall u \in L^p(\O)\,, u\geq 0\,,  \int_\O u^p=\int_\O(u^\#)^p.
\end{equation}
Of great importance to us are two inequalities. The first one, the so-called Poly\'a-Szeg\"{o} inequality asserts that 
\begin{equation}
\forall \p\in W^{1,2}_0(\O)\,, \p^\#\in W^{1,2}_0(\O)\text{ and }\int_\O |\n \p^\#|^2\leq \int_\O |\n \p|^2.\end{equation}
The equality case in this equality was fully derived in \cite{Brothers1988} (see also \cite{Ferone2003}), and quantitative versions were given in \cite{Barchiesi2014,CianchiEsposito}. The second one is the Hardy-Littlewood inequality:
\begin{equation}\forall f\,, g\in L^2(\O)\,, f\,, g\geq 0\,, \int_\O fg\leq \int_\O f^\#g^\#.\end{equation} This inequality can be rewritten in the following form: for a.e. $\tau$, 
\begin{equation}\int_{\{g>\tau\}}f\leq \int_{\{g^\#>\tau\}}f^\#.\end{equation}
A quantitative version of this inequality can be found in \cite{Cianchi2008} (and \cite{MRB2020} in a simpler case where smoothness of the involved function $f$ is assumed). In Propositions \ref{Pr:Bathtub} and \ref{Pr:BathtubTD}, we give uniform versions of this quantitative inequality for families of functions.

Comparison principle for parabolic equations started with the work of Bandle \cite{Bandle}, Vazquez \cite{Vazquez}, using the seminal ideas of Talenti \cite{Talenti}, and were later extended in a series of works by Alvino, Lions and Trombetti \cite{Alvino1986,Alvino1990} and Rakotoson and Mossino \cite{RakotosonMossino}. By "comparison principle for parabolic equations" we mean results that enable one to compare the solution $u$ of a parabolic equation of the form
\begin{equation}\frac{\partial u}{\partial t}-\Delta u=f(t,\cdot)\end{equation} with the solution $v$ of the symmetrised equation 
\begin{equation}\frac{\partial v}{\partial v}-\Delta v=f^\#(t,\cdot).\end{equation} Both equations are supplemented with Dirichlet boundary conditions, and we wilfully ignore first order terms. The correct comparison relation $\prec$ used for such comparisons is defined as:
\begin{equation}\label{Eq:Prec}\text{$f\prec g$ if and only if for any }
r \in [0;R]\,, \int_{\mathbb B(0;r)} f^\#\leq \int_{\mathbb B(0;r)} g,
\end{equation}
and the typical result asserts that $u^\#(t,\cdot)\prec v(t,\cdot)$. In this paper, we will rely, for the uniqueness result, Theorem \ref{Th:Uniqueness}, on the method of proof of \cite{RakotosonMossino}, which enables more easily to encompass the equality case. We expand on their techniques in the proof of Theorem \ref{Th:Uniqueness}, see Section \ref{Se:Uniq}. 

 \subsection{Plan of the paper}\label{Se:Plan}This paper is structured as follows:
 \begin{enumerate}
 \item In Section \ref{Se:Preliminary} we gather several elementary information about the optimisation problems (adjoint, switch function, regularity of the solutions, convexity of the functionals).
 \item In Section \ref{Se:Uniq}, we prove the uniqueness result stated in Theorem \ref{Th:Uniqueness}.
 \item Section \ref{Se:Td} contains the proof of Theorem \ref{Th:Td} and is independent of Section \ref{Se:Ti}.
  \item Section \ref{Se:Ti} corresponds to the proof of Theorem \ref{Th:Ti}. In it, we state our coercivity results for second order shape derivatives. This Section is independent of Section \ref{Se:Td}.
 \item The Conclusion, Section \ref{Se:Concl}, contains discussion about possible extensions, as well obstructions for generalising the results presented here.
 \end{enumerate}
 
\subsection{Notational conventions}\label{Su:Notation}
\begin{itemize}
\item For any $g\in L^2(\O)$, $g^\#$ denotes its Schwarz rearrangement and $g^\dagger$ its one-dimensional counterpart.
\item $\B^*=\mathbb B(0;r^*)$ is the unique centred ball of volume $V_0$. In other words, it is the only centred ball satisfying $\mathds1_{\B^*}\in \overline{\mathcal M}(\O).$
\item $u^*$ is the solution of \eqref{Eq:Main} associated with the static control $f\equiv f^*=\mathds 1_{B^*}$.
\item For a function $f$ that is discontinuous across a smooth hypersurface $\Sigma$ with oriented normal $\nu$, but continuous in $\O\backslash \Sigma$, the jump of $f$ across $\Sigma$ is 
\begin{equation}\left.\llbracket f\rrbracket\right|_\Sigma:=\lim_{t\to 0^+}\left(f(x+t\nu(x))-f(x-t\nu(x))\right).\end{equation}
\end{itemize}

\section{Preliminary results}\label{Se:Preliminary}
We gather here several results that will be used throughout the paper. We begin with some basic regularity estimates on the solutions of the equation. 
\begin{proposition}\label{Pr:Regularity}
For any $\alpha\in (0;1)$, there exists $M_\alpha>0$ such that, for any $f\in \mathcal M_T(\O)$, we have the estimate 
\begin{equation}
\Vert u_f(t,\cdot)\Vert_{\mathscr C^{0,\alpha}(\OT)}\leq M_\alpha.
\end{equation}

Furthermore, for any $\alpha\in (0;1)$ and almost every $t\in (0;T)$, $u_f(t,\cdot)\in \mathscr C^{1,\alpha}(\O)$.
\end{proposition}

\begin{proof}[Proof of Proposition \ref{Pr:Regularity}]

For the first point, we use \cite[Corollary 7.31, p.182]{Lieberman} which ensures that for any $p>1$,
\begin{equation}\int_0^T \Vert u(t,\cdot)\Vert_{W^{2,p}(\O)}+\left\Vert \frac{\partial u}{\partial t}\right\Vert_{L^p(\O)}\leq C(\Vert f\Vert_{L^\infty}+\Vert u^0\Vert_{\mathscr C^2}),\end{equation} where $C$ depends on the dimension, on $p$ and on $\O$. It thus follows that, in particular, for any $p\in (1;+\infty)$ there exists $C_p$ such that 
\begin{equation}\Vert u\Vert_{W^{1,p}(\OT)}\leq C_p.\end{equation} It then suffices to apply the Sobolev embedding $W^{1,p}(\OT)\hookrightarrow
\mathscr C^{0,\alpha}(\O)$ for $p$ large enough.

The second point follows from the fact that, from the same estimate, for any $p>1$ and almost every $t\in (0;T)$, $u_f(t,\cdot)\in W^{2,p}(\O)$. The conclusion follows by the Sobolev embedding $W^{2,p}(\O)\hookrightarrow \mathscr C^{1,\alpha}(\O)$.
\end{proof}

We then provide structural information about the functionals which we seek to optimise. In Proposition \ref{Pr:Convexity}, we establish convexity properties which will prove crucial while, in Proposition \ref{Pr:Adjoint}, we compute the adjoint and the switch function of the equation.
 
\begin{proposition}\label{Pr:Convexity}
The map $\mathcal J_T:\overline{\mathcal M}(\O)\ni f\mapsto \mathcal J_T(f)$ is strictly convex. In the same way, for any $\e> 0$, the map $\mathcal J_T^\e:\mathcal M_T(\O)\ni f\mapsto \mathcal J_T^\e(f)$ is strictly convex.

\end{proposition}
\begin{proof}[Proof of Proposition \ref{Pr:Convexity}]
We only prove the convexity of $\mathcal J_T$, the convexity of $\mathcal J_T^\e$ following along the same lines.

It follows from standard argument that the map $\overline{\mathcal M}(\O)\ni f\mapsto u_f$ is twice G\^ateaux-differentiable. The convexity of the functional is equivalent to requiring that the second order G\^ateaux derivative be non-negative. For any admissible perturbation $h$ at $f$ (that is, such that for every $t>0$ small enough $f+th\in \overline{\mathcal M}(\O)$) the G\^ateaux-derivative of $u_f$ in the direction $h$, denoted by $\dot u_f$, solves 
\begin{equation}\label{Eq:Chanel}
\begin{cases}
\frac{\partial \dot u_f}{\partial t}-\Delta \dot u_f=h\text{ in }\OT\,, \\ \dot u_f=0\text{ on }(0;T)\times \partial \O\,, \\\dot u_f(0,\cdot)\equiv 0.\end{cases}\end{equation}
From this equation on $\dot u_f$, we deduce that the G\^ateaux-derivative of $\mathcal J_T$ at $f$ in the direction $h$ is given by 
\begin{equation}\label{Eq:KL}
\dot{\mathcal J}_T(f)[h]=\iint_\OT \dot u_f u_f.\end{equation}

In the same way, the second order G\^ateaux-derivative of $u_f$ in the direction $h$, denoted by $\ddot u_f$, satisfies
\begin{equation}
\begin{cases}
\frac{\partial \ddot u_f}{\partial t}-\Delta \ddot u_f=0\text{ in }\OT\,, \\ \ddot u_f=0\text{ on }(0;T)\times \partial \O\,, \\\ddot u_f(0,\cdot)\equiv 0\end{cases}\end{equation} and thus
\begin{equation}\ddot u_f=0.\end{equation} 
Furthermore, the second order G\^ateaux-derivative of $\mathcal J_T$ in the direction $h$, denoted by $\ddot{\mathcal J_T}(f)[h,h]$, is given by 
\begin{equation}\ddot{\mathcal J_T}(f)[h,h]=\iint_\OT \left(\dot u_f\right)^2+\iint_\OT\ddot u_f u_f=\iint_\OT \left( \dot u_f\right)^2\geq 0\end{equation} and the last inequality is strict unless $h\equiv 0.$ Since the second-order G\^ateaux-derivative of the functional is non-negative, the functional is convex.
\end{proof}
This convexity property is one of the fundamental point to carry out the proof of Theorem \ref{Th:Ti}.

\begin{proposition}\label{Pr:Adjoint}Let $f\in \overline{ \mathcal M}(\O)$. Let $p_f$ be the unique solution of 
\begin{equation}\label{Eq:Fendi}\begin{cases}
\frac{\partial p_f}{\partial t}+\Delta p_f=-u_f\text{ in }\OT,
\\ p_f(T,\cdot)=0\,, 
\\ p_f(t,\cdot)=0\text{ on }(0;T)\times \partial \O.\end{cases}\end{equation}
Then for any $f\in \overline{\mathcal M}(\O)$ and any admissible perturbation $h$ at $f$, the G\^ateaux-derivative of $\mathcal J_T$ at $f$ in the direction $h$ is given by  
\begin{equation}\label{Eq:Coco}\dot{\mathcal J}_T(f)[h]=\iint_\OT h(x)p_f(t,x)dtdx.\end{equation} In the same way, let us consider a parameter $\e>0$. Let $f\in \mathcal M_T(\O)$ and define $p_{\e,f}$ as the unique solution of 
\begin{equation}\begin{cases}
\frac{\partial p_{\e,f}}{\partial t}+\Delta p_{\e,f}=-u_f\text{ in }\OT,
\\ p_{\e,f}(T,\cdot)=\e u_f(T,\cdot)\,, 
\\ p_{\e,f}(t,\cdot)=0\text{ on }(0;T)\times \partial \O.\end{cases}\end{equation} Then for any $f\in {\mathcal M_T}(\O)$ and any admissible perturbation $h$ at $f$, the G\^ateaux-derivative of $\mathcal J_T^\e$ at $f$ in the direction $h$ is given by  
\begin{equation}\dot{\mathcal J}_T^\e(f)[h]=\iint_\OT h(t,x)p_f(t,x)dtdx.\end{equation}
$p_f$ is dubbed the switch function for the functional $\mathcal J_T$, while $p_{\e,f}$ is dubbed the switch function for the functional $\mathcal J_T^\e$.
\end{proposition}

\begin{proof}[Proof of Proposition \ref{Pr:Adjoint}]
We only prove this proposition in the case $f\in \overline{\mathcal M}(\O)$, the time-dependent case following along the same exact lines. Let us first note that, as a backward, linear heat equation, existence and uniqueness of a solution to \eqref{Eq:Fendi} is guaranteed.

To get \eqref{Eq:Coco}, we start from the expression \eqref{Eq:KL} of the first order G\^ateaux-derivative of the functional $\mathcal J_T$:
\begin{equation}\dot{\mathcal J}_T(f)[h]=\iint_\OT \dot u_f u_f,\end{equation} where $\dot u_f$ solves \eqref{Eq:Chanel}. If we multiply this equation by the solution $p_f$ of \eqref{Eq:Fendi} and integrate by parts, we get 
\begin{equation}
\iint_\OT \dot u_f u_f=\iint_\OT h p_f.\end{equation} Since $\dot{\mathcal J}_T(f)[h]=\iint_\OT \dot u_f u_f,$ the conclusion follows.

\end{proof}
 We conclude this section with some information about the function $u^*$ solution of \eqref{Eq:Main} with $f\equiv f^*$.

\begin{proposition}\label{Pr:Radial}
The solution $u^*$ of \eqref{Eq:Main} with $f\equiv f^*$ is radially symmetric. Furthermore, for any $r\in (0;R)$ and any $t\in (0;T)$
\begin{equation}
-\frac{\partial u^*}{\partial r}(t,r)>0.
\end{equation}
\end{proposition}
\begin{proof}[Proof of Proposition \ref{Pr:Radial}]
The radial symmetry of the solution is immediate. In radial coordinates, and with a slight abuse of notation, $u^*$ satisfies
\begin{equation}\label{Eq:MainEtoile}
\begin{cases}
\frac{\partial u^*}{\partial t}-\frac1{r^{n-1}}\frac{\partial }{\partial r}\left({r^{n-1}}\frac{\partial u^*}{\partial r}\right)=f^*\text{ in } (0;T)\times (0;R)\,, 
\\ u^*(t,R)=\frac{\partial u^*}{\partial r}(t,0)=0\text{ for any }t\,, 
\\ u^*(0,\cdot)=u^0.
\end{cases}
\end{equation}
From Proposition \ref{Pr:Regularity} above we can differentiate $u^*$ with respect to $r$. Let us write 
\begin{equation}z:=\frac{\partial u}{\partial r}.\end{equation} It follows from \eqref{Eq:MainEtoile} that $z$ solves 
\begin{equation}\frac{\partial z}{\partial t}-\frac1{r^{n-1}}\frac{\partial  }{\partial r}\left({r^{n-1}}\frac{\partial z}{\partial r}\right)=-\frac{(n-1)z}{r^2}\text{ in } (0;T)\times (0;R)\end{equation} and that the following jump condition is satisfied at $r=r^*$:
\begin{equation}\left\llbracket \frac{\partial z}{\partial r}\right\rrbracket(t,r^*)=-\left\llbracket f^*\right\rrbracket(r^*)=1>0.\end{equation} Since $u^*\geq 0$, the parabolic Hopf Lemma implies that for any $t>0$, 
\begin{equation}
z(t,R)< 0.\end{equation}Differentiating \eqref{Eq:MainEtoile} with respect to $r$ and remembering that $u^0$ is non-increasing, we obtain that $z$ solves 
\begin{equation}\label{Eq:z}
\begin{cases}
\frac{\partial z}{\partial t}-\frac1{r^{n-1}}\frac{\partial  }{\partial r}\left(r^{n-1}\frac{\partial z}{\partial r}\right)=-\frac{(n-1)z}{r^2}\text{ in } (0;T)\times (0;R)\,, 
\\ z(t,R)<0\text{ for any }t\,, 
\\ z(t,0)=0,
\\ \left\llbracket \frac{\partial z}{\partial r}\right\rrbracket(t,r^*)=1\,, t\in (0;T),
\\ z(0,\cdot)\leq0.
\end{cases}
\end{equation}
Let us then show that for any $(t,r)\in (0;T)\times (0;R)$
\begin{equation}z(t,r)<0.\end{equation}
 First of all, multiplying \eqref{Eq:z} by the positive part $z_+$ of $z$ we get (keeping in mind that $z_+(t,R)=0$)
\begin{equation}
\frac12\frac\partial{\partial t}\int_0^R r^{n-1} z_+(t,r)^2dr+\int_0^R r^{n-1}\left(\frac{\partial z_+}{\partial r}\right)^2+(r^*)^{n-1}z_+(t,r^*)=-\int_0^R \frac{(n-1)z_+(t,r)^2}{r^{3-n}}dr
\end{equation}  so that  $z_+(t,\cdot)=0$. As a consquence, $z\leq 0$. To argue that $z<0$ in $(0;T)\times \O$, we follow the same procedure as for the strong parabolic maximum principle. If we first assume $z\leq 0$ satisfies an inequality rather than an equality, that is, that $z$ satisfies
$$\frac{\partial z}{\partial t}-\frac1{r^{n-1}}\frac{\partial  }{\partial r}\left(r^{n-1}\frac{\partial z}{\partial r}\right)<-\frac{(n-1)z}{r^2}$$ then if by contradiction we assume that there exists $(t_0,r_0)$, with $r_0\in (0;R]$ and $t_0\in (0;T)$ such that $z(t_0,r_0)=0$, it follows that $r_0<R$. By the jump condition, $r_0\neq r^*$. Since $t_0<T$, plugging the optimality conditions, the contradiction follows.  To exclude the case $t_0=T$ it suffices to consider the equation on $[0;T+\epsilon]\,, \epsilon>0$ and to carry out the same reasoning in $(0;T+\epsilon)$. To then pass from this case (strict inequality) to ours (the equality case), with 
$$\frac{\partial z}{\partial t}-\frac1{r^{n-1}}\frac{\partial  }{\partial r}\left(r^{n-1}\frac{\partial z}{\partial r}\right)=-\frac{(n-1)z}{r^2}$$ it suffices to consider $z_\e(t,r):=z(t,r)-\e t$ and the conclusion follows from passing to the limit $\e\to 0^+$.
\end{proof}

\section{Proof of Theorem \ref{Th:Uniqueness}: Uniqueness of maximisers}\label{Se:Uniq}
\begin{proof}[Proof of Theorem \ref{Th:Uniqueness}]
It follows from \cite{RakotosonMossino} that $f^*$ is a solution of \eqref{Eq:PvTi} and \eqref{Eq:PvTd}. The uniqueness property for \eqref{Eq:PvTd} implies uniqueness for \eqref{Eq:PvTi} so we focus on the time-dependent case. Let us define ${u^*}$ as the solution of \eqref{Eq:Main} associated with $f\equiv f^*$. We consider another solution $f$ of \eqref{Eq:PvTd} and the solution $u$ of \eqref{Eq:Main} associated. By convexity of the functional we can assume that $f$ is a characteristic function so that \begin{equation}f^\#=f^*.\end{equation}

We proceed along a series of claims. The first one is :
\begin{claim}\label{Cl:RM1}
If $f$ solves \eqref{Eq:PvTd} and if $u$ is the associated solution of \eqref{Eq:Main} then for almost every $t\in (0;T)$, there holds 
\begin{equation}
u^\#(t,\cdot)={u^*}(t,\cdot).\end{equation}
\end{claim}
\begin{proof}[Proof of Claim \ref{Cl:RM1}]
It follows from \cite{RakotosonMossino} and the results recalled in the introduction that for almost every $t\in (0;T)$ we have 
\begin{equation}\label{Eq:Int}
u^\#(t,\cdot)\prec {u^*}(t,\cdot).\end{equation} 
The relation $\prec$ was defined in Equation \eqref{Eq:Prec}. Thus, from \cite[Proposition 2]{alvino1991} we have that for almost every $t\in (0;T)$ we have 
\begin{equation}(u^\#)^2(t,\cdot)\prec {(u^*)}^2(t,\cdot).\end{equation} Integrating this inequality in time and in space yields
\begin{equation}\iint_{(0;T)\times \O} \left(u^\#\right)^2\leq \iint_{(0;T)\times \O} \left({u^*}\right)^2.\end{equation} However, by equimeasurability of the Schwarz rearrangement \eqref{Eq:Equimeasurability} we have
\begin{equation}
\iint_{(0;T)\times \O}u^2=\iint_{(0;T)\times \O} (u^\#)^2.\end{equation} 
Since $f$ is a maximiser of \eqref{Eq:PvTd} it follows that equality holds for almost every $t$ in 
\begin{equation}\int_\O \left(u^\#\right)^2(t,\cdot)\leq \int_\O \left({u^*}\right)^2(t,\cdot).\end{equation} Thus we have for almost every $t\in (0;T)$, 
\begin{equation}\label{Eq:Pa}\int_\O \left(u^\#\right)^2(t,\cdot)= \int_\O \left({u^*}\right)^2(t,\cdot).\end{equation} 
Let us now introduce the set $\mathscr K({u^*})$ defined as
\begin{equation}\mathscr K({u^*})=\left\{g\in L^2(\O)\,, g\prec {u^*}\right\}.\end{equation}
From \cite{Alvino1989} this is a compact (for the weak $L^\infty-*$ topology) and convex set whose set of extreme points is 
\begin{equation}\mathscr C({u^*})=\left\{g\in L^2(\O)\,, g^\#={u^*}\right\}.\end{equation}
Since $x\mapsto x^2$ is strictly convex, the map $\mathscr K(v)\ni g\mapsto \int_\O g^2$ is strictly convex. Besides, once again because of the convexity of $x\mapsto x^2$, we have, for any $g\in \mathscr K(u^*)$, $$(g^\#)^2=(g^2)^\#.$$
 As a consequence, the only solutions of the maximisation problem
\begin{equation}\label{Eq:LS}
\sup_{g\in \mathscr K({u^*})}\int_\O g^2
\end{equation} are exactly the elements of $\mathscr C({u^*})$. 

On the other hand, \eqref{Eq:Pa} states that $u(t,\cdot)$ is a solution of \eqref{Eq:LS}, so it follows that for almost every $t\in (0;T)$ there holds 
\begin{equation}
u^\#(t,\cdot)=u^*(t,\cdot).\end{equation}

\end{proof}
In particular, and this is the main point of this proof, the two following properties hold: first,
\begin{equation}\label{Eq:Simp} \text{ If $f$ solves \eqref{Eq:PvTd} then for a.e. $t\in (0;T)$, } u^\dagger(t,\cdot)=(u^*)^\dagger(t,\cdot).\end{equation} Second, we have, as a consequence the following fact:
\begin{equation}\text{If $f$ solves \eqref{Eq:PvTd} then for a.e. $t\in (0;T)$, }\int_\O u(t,\cdot)=\int_\O u^\#(t,\cdot)=\int_\O u^*(t,\cdot).\end{equation}

We then  prove that if $f$ solves \eqref{Eq:PvTd}, then all the level sets of $u$ are balls.
\begin{claim}\label{Cl:Balls}
If $f$ solves \eqref{Eq:PvTd}, then all the level sets of $u$ are balls.
\end{claim}

\begin{proof}[Proof of Claim \ref{Cl:Balls}]

We follow the approach of \cite{RakotosonMossino}. We first recall \cite[Theorem 1.2]{RakotosonMossino}: if $\p\in W^{1,2}((0;T),L^2(\O))$ then $\p^\#\in W^{1,2}((0;T),L^2(\O))$ and moreover there holds, if $\p$ only has measure sets of measure zero,
\begin{equation}\label{Eq:RakoMossi}
\frac{\partial \p^\#}{\partial t}(t,s)=\frac{\partial w}{\partial s}(t,s)
\end{equation}
where $w$ is defined by
\begin{equation}
w(t,s)=\int_{\left\{\p(t,\cdot)\leq \p^\#(t,s)\right\}}\frac{\partial \p}{\partial t}.
\end{equation}

We then consider \eqref{Eq:Main}. For any $\tau \in \R_+$, we multiply the equation by $(u-\tau)_+$ and integrate by parts in space. We obtain in a classical way
\begin{equation}\label{Eq:Lis}
0\leq -\frac{\partial}{\partial \tau}\int_{\{u>\tau\}} |\n u|^2(t,\cdot)=\int_{\{u>\tau\}} \left(f-\frac{\partial u}{\partial t}(t,\cdot)\right).\end{equation} We write the repartition function of $u$ as $\mu$:
\begin{equation}\mu(t,\tau)=\operatorname{Vol}\left(\{u(t,\cdot)>\tau\}\right).\end{equation}By the isoperimetric inequality and the co-area formula, taking $S_n:=n\operatorname{Vol}(\mathbb B(0;1))^{\frac1n}$, we obtain, as in \cite{RakotosonMossino},
\begin{align}\label{Eq:Co}
S_n 
\mu(t,\tau)^{1-\frac1n}&\leq \left(-\frac{\partial }{\partial \tau}\int_{\{u(t,\cdot)>\tau\}}|\n u|\right)
\\&\leq \left(-\frac{\partial \mu}{\partial \tau}\right)^{\frac12}\left(-\frac{\partial}{\partial \tau}\int_{\{u(t,\cdot)>\tau\}}|\n u|^2\right)^{\frac12}.
\end{align}
This leads to
\begin{align}
S_n \mu(t,\tau)^{1-\frac1n}&\leq\left(-\frac{\partial \mu}{\partial \tau}\right)^{\frac12}\left(-\frac{\partial}{\partial \tau}\int_{\{u(t,\cdot)>\tau\}}|\n u|^2\right)^{\frac12}
\\&\leq \left(-\frac{\partial \mu}{\partial \tau}\right)^{\frac12}\left(\int_{\{u(t,\cdot)>\tau\}} \left(f-\frac{\partial u}{\partial t}(t,\cdot)\right)\right)^{\frac12}.
\end{align}
Hence,

\begin{equation}\label{Eq:Viv}
S_n^2\mu(t,\tau)^{2-\frac2n}\leq \left(-\frac{\partial \mu}{\partial \tau}\right)\int_{\{u(t,\cdot)>\tau\}} \left(f-\frac{\partial u}{\partial t}\right).\end{equation} Here we recall that $\int_{\{u(t,\cdot)>\tau\}} \left(f-\frac{\partial u}{\partial t}\right)\geq 0$ by \eqref{Eq:Lis}.

As is customary we use the Hardy-Littlewood inequality to obtain 
\begin{equation}\label{Eq:Hl}
\int_{\{u(t,\cdot)>\tau\}}f\leq \int_0^{\mu(t,\tau)}f^\dagger=:F(t,\mu(t,\tau)).\end{equation} 
Let us now define 
\begin{equation}k(t,\tau):=\int_0^\tau u^\dagger(t,s)ds\end{equation} and we obtain 

\begin{equation}\int_{\{u(t,\cdot)>\tau\}}\frac{\partial u}{\partial t}=\frac{\partial k}{\partial t}(t,\mu(t,\tau)).\end{equation}

As such, for some constant $c_n>0$,

\begin{equation}
1\leq S_n^{-2}\left({-\frac{\partial \mu}{\partial \tau}}\right)\mu(t,\tau)^{\frac{2}n-2}\left(F(t,\mu(t,\tau))-\frac{\partial k}{\partial t}(t,\mu(t,\tau))\right).
\end{equation}
Integrating this equation between $\tau_0$ and $\tau_1$ for any $0\leq \tau_0\leq \tau_1$ yields
\begin{equation}
\tau_1-\tau_0\leq S_n^{-2}\int_{\mu(t,\tau_0)}^{\mu(t,\tau_1)}s^{-2+\frac2n} \left(F(t,s)-\frac{\partial k}{\partial t}(t,s)\right)ds.\end{equation} We hence get in a classical way \cite{Mossino} the following differential inequality
\begin{equation}
-\frac{\partial u^\dagger}{\partial \tau}(t,\tau)=
-\frac{\partial^2 k}{\partial \tau^2}(t,\tau)\leq  S_n^{-2}\tau^{-2+\frac2n}\left(F(t,\tau)-\frac{\partial k}{\partial t}(t,\tau)\right).
\end{equation}
Let us now define \begin{equation}
k_{u^*}(t,\tau):=\int_0^\tau \left(u^*\right)^\dagger(t,\cdot).\end{equation}
 We recall that $u^*$ is the solution of \eqref{Eq:Main} associated with $f\equiv f^*$. Since $f$ is radially symmetric and decreasing, all the equalities in the above reasoning carried for $u$ hold for $u^*$ with equalities instead of inequalities and $k_{u^*}$ solves
\begin{equation}
\frac{\partial^2 k_{u^*}}{\partial \tau^2}+S_n^{-2}\tau^{-2+\frac2n}\frac{\partial k_{u^*}}{\partial t}= S_n^{-2}\tau^{-2+\frac2n}F(t,\tau).\end{equation}

Finally, we set $K=k-k_{u^*}$. From Equation \eqref{Eq:Simp}, we have, for any $t\in (0;T)$ and any $s\in (0;\operatorname{Vol}(\O))$,
\begin{equation} K(t,s)=0.\end{equation} 
Since $K\equiv 0$, every equality in the above reasoning must in fact be an equality. In particular, \eqref{Eq:Co} is an equality, and hence all the level-sets of $u$ are balls, which concludes the proof.
\end{proof}
\begin{remark}
It would be interesting to investigate whether or not using the quantitative isoperimetric inequality could lead to quantitative estimates, but it is not at this point clear how to do that. We refer to the Conclusion, Section \ref{SD}.\end{remark}
As is customary in the study of equality cases in Talenti-like inequalities, we need to check that the level sets are not just balls but rather concentric balls.
\begin{claim}\label{Cl:BouleConcentrique}
If $f$ solves \eqref{Eq:PvTd} then the level sets of the associated solution $u$ are concentric balls.\end{claim}
\begin{proof}[Proof of Claim \ref{Cl:BouleConcentrique}]
The core idea of the proof is similar to \cite{Kesavan1988}. Let us first consider the solution $w$ of 
\begin{equation}\label{Eq:w}
\begin{cases}
\frac{\partial w}{\partial t}+\Delta w=-1\text{ in }(0;T)\times \O\,, 
\\ w=0\text{ on }(0;T)\times \O\,, 
\\ w(T,\cdot)=0.
\end{cases}\end{equation}
It follows from the same arguments as in the proof of Proposition \ref{Pr:Radial} that $w$ is radially symmetric and decreasing (for any $t<T$),  and so we obtain by the Hardy-Littlewood inequality that for almost every $t\in (0;T)$,
\begin{equation}
\int_\O fw \leq \int_\O f^\#w=\int_\O f^*w.\end{equation}

However multiplying Equation \eqref{Eq:w} by $u$ and integrating by parts both in time and space yields
\begin{align*}
\iint_\OT fw&=\iint_\OT \left(\frac{\partial u}{\partial t}-\Delta u\right)w
\\&=-\int_\O w u^0-\iint_\OT u\left(\frac{\partial w}{\partial t}+\Delta w\right)
\\&=-\int_\O wu^0+\iint_\OT u
\\&=-\int_\O wu^0+\iint_\OT u^* \text{  because of Claim \ref{Cl:RM1}}
\\&=\iint_\OT f^* w \text{ by the same computations with $u^*$ instead of $u$}.
\end{align*}
However, and since $w$ is radially symmetric and increasing, the Hardy-Littlewood inequality implies that for almost every $t\in (0;T)$ and almost every $\tau$ we have 
\begin{equation}\label{Eq:Kla}
 \int_{\{w(t,\cdot)>\tau\}} f\leq  \int_{\{w(t,\cdot)>\tau\}} f^\#.\end{equation} Hence it follows that \eqref{Eq:Kla} must be an equality for almost every $t$.
Thus since for almost every $t$ the function $w$ is symmetric and radially decreasing we get 
\begin{equation}\label{Eq:Ke}
\forall r \in (0;R)\,, \int_{\mathbb B(0;r)} f=\int_{\mathbb B(0;r)} f^*.\end{equation}

For the final step, let $\phi_1$ be the first Dirichlet eigenvalue of the laplacian in $\O$. It is standard to see that $\phi_1$ is radially symmetric and decreasing. Introduce the solution $\phi$ of 
\begin{equation}\label{Eq:phi}
\begin{cases}
\frac{\partial \phi}{\partial t}+\Delta \phi=-\phi_1\text{ in }(0;T)\times \O\,, 
\\ \phi=0\text{ on }(0;T)\times \O\,, 
\\ \phi(T,\cdot)=0.
\end{cases}\end{equation} The function $\phi$ is radially symmetric and decreasing as well for any $t<T$. As a consequence, all level-sets of $\phi(t,\cdot)$ are level-sets of $w(t,\cdot)$ and conversely, from which we deduce that, for almost every $t\in (0;T)$ and almost every $\tau$
\begin{equation}\int_{\{\phi(t,\cdot)>\tau\}}f=\int_{\{\phi(t,\cdot)>\tau\}}f^\#=\int_{\{\phi(t,\cdot)>\tau\}}f^*.\end{equation} This gives in turn
\begin{equation}\int_\O f\phi(t,\cdot)=\int_\O f^* \phi(t,\cdot).\end{equation}

 Multiplying \eqref{Eq:phi}  by $u$ and integrating by parts gives in the same way
\begin{equation}
\iint_\OT u\phi_1=\iint_\OT f\phi=\iint_\OT f^*\phi= \iint_\OT u^*\phi_1=\iint_\OT u^\# \phi_1.
\end{equation}
The last equality comes from \eqref{Eq:Simp}.

Invoking the Hardy-Littlewood inequality we obtain in the same fashion that for almost every $t\in (0;T)$
\begin{equation}
\forall r\in (0;R)\,, \int_{\mathbb B(0;r)} u=\int_{\mathbb B(0;r)} u^\#.\end{equation} It follows that $u=u^\#$ so that the conclusion is reached.

\end{proof}

\end{proof}

\section{Proof of Theorem \ref{Th:Td}}\label{Se:Td}
\subsection{Plan of the proof and heuristics}
This theorem relies on the following fact: assuming that we have a competitor $f$, to be compared with $f^*$, and defining, for every $t\in [0;T]$, 
\begin{equation}\delta(t):=\Vert f(t,\cdot)-f^*\Vert_{L^1(\O)}^2\end{equation} we can set
\begin{equation}\label{Eq:MTDelta}\mathcal M_T(\O,\delta)=\left\{g\in \mathcal M_T(\O)\,, \text{ for a.e. }t\in [0;T]\,, \Vert g(t,\cdot)-f^*\Vert_{L^1(\O)}=\delta(t)\right\}\end{equation} and replace $f$ with the solution $f_\delta^*$ of 

\begin{equation}\label{Eq:DeltaTD}
\max_{f\in \mathcal M_T(\O,\delta)}\mathcal J_T^\e(f).
\end{equation}
That such a solution exists follows by the same argument as in Lemma \ref{Le:ExistenceDelta} below (see the proof in Appendix \ref{Ap:Technical}) but we can actually prove (and this is the part that is specific to $\O$ being a centred ball) that the solutions to \eqref{Eq:DeltaTD}  admits the following explicit description: let, for any $\overline \delta>0$, $\A_{\overline\delta}$ be defined, in radial coordinates, as
\begin{equation}
\A_{\overline\delta}=\{r<r^*-r_{\overline\delta}^-\}\sqcup \{r^*<r<r^*+r_{\overline\delta}^+\}\end{equation} where $r_{\overline\delta}^-,r_{\overline\delta}^+$ are the unique parameters such that 

\begin{equation}\operatorname{Vol}(\A_{\overline \delta})=V_0\,, \operatorname{Vol}\left(\A_{\overline \delta}\Delta \B^*\right)=\overline \delta.\end{equation} Then we will show (Proposition \ref{Pr:GoodGuys}) 
\begin{equation}\label{Eq:Fdelta}f_\delta:t\mapsto \mathds 1_{\mathbb A_{\delta(t)}}\end{equation} is a solution of \eqref{Eq:DeltaTD}. Throughout the rest of this introduction to the proof, we keep the notation $f_\delta$ for this function.

Let us formally assume that
\begin{equation}\int_0^T \delta(t)dt\ll 1\end{equation} and define, for any $\xi\in (0;1)$, $p_{\e,\xi}$ the adjoint state associated with $f(t)=f^*+\xi \left(f_\delta-f^*\right)$. By parabolic regularity, $p_{\e,\xi}$ should be a non-increasing function of $r$ since the adjoint state $p^*_\e$ associated to $f^*$ is decreasing. By the mean-value theorem, there exists $\xi \in [0;1]$ such that 
\begin{equation}
\mathcal J_T^\e(f_\delta)-\mathcal J_T^\e(f^*)=\iint_\OT p_{\e,\xi}\left(f_\delta-f^*\right). \end{equation} A natural step is then to try and apply the quantitative bathtub principle to this quantity: since $p_{\e,\xi}$ is a radially symmetric, non-increasing function of $r$, then for any $t\in (0;T)$, $f^*$ is the only solution of 
\begin{equation}
\sup_{f \in \overline{\mathcal M}(\O)}\int_\O fp_{\e,\xi}(t,\cdot).
\end{equation}
The hope is then to prove that there exists a constant $C>0$ such that for any $t\in (0;T)$ there holds
\begin{equation}\forall f \in \overline{\mathcal M}(\O)\,, \int_\O (f-f^*)p_{\e,\xi}\leq- C\left(-\frac{\partial p_{\e,\xi}}{\partial r}\right)(t,r^*)\Vert f(t,\cdot)-f^*\Vert_{L^1(\O)}^2.\end{equation}
However, the existence of such a uniform constant relies, in a crucial way, on $\e$: when $\e>0$, it is possible while, when $\e=0$, other difficulties may arise. The key difficulty is that when $\e>0$ we can guarantee that 
\begin{equation}\sup_{t\in [0;T]}\frac{\partial p_\e^*}{\partial r}(t,r^*)<0\end{equation} while for $\e=0$ we can only guarantee 

\begin{equation}\forall \tau>0\,, \exists \alpha(\tau)>0\,, \sup_{t\in [0;T-\tau]}\frac{\partial p_\e^*}{\partial r}(t,r^*)\leq-\alpha(\tau).\end{equation} 

To give a synthetic presentation, we isolate the main tool of this proof in the following paragraph.

\subsection{Uniform quantitative bathtub principle}

\begin{proposition}\label{Pr:BathtubTD}
Let  $\beta>0$ and consider a family of function $\{p_i\}_{i\in I}\in \mathscr C^{1,\beta}(\O)$ such that:
\begin{enumerate}
\item There exists $M>0$ such that 
\begin{equation}\label{Eq:HolderUniforme}
\sup_{i\in I}\Vert p_i\Vert_{\mathscr C^{1,\beta}}\leq M.\end{equation}
\item For any $i\in I$, $p_i$ is radially symmetric.  Furthermore, there exists $\alpha>0$ such that, for any $r\in [0;r^*]$, 
\begin{equation}\label{Eq:Mack}\forall i \in I\,, p_i(r)-p_i(r^*)\geq \alpha |r-r^*|.\end{equation} We also assume that for any $i\in I$, $p_i$ is decreasing in $(r^*;R)$. In particular, the unique level set of $p_i$ of volume $V_0$ is $\mathbb B(0;r^*)$: there exists $c_i$ such that 
\begin{equation}\B(0;r^*)=\{p_i>c_i\}\,, \partial \B(0;r^*)=\{p_i=c_i\}.\end{equation} This in particular ensures that the minimum of $p_i$ in $\mathbb B(0;r^*)$ is only achieved on $\partial \B(0;r^*)$.  As another consequence, for this constant $\alpha>0$,  we have
\begin{equation}\label{Eq:DerivUniforme}
\forall i\in I\,, -\frac{\partial p_i}{\partial r}(r^*)\geq \alpha>0.
\end{equation}
\end{enumerate}
Then there exists a constant $\omega>0$ such that
\begin{equation}\label{Eq:BathtubTD}
\forall f \in \mathcal M(\O)\,, \forall i\in I\,, \int_\O p_i(f^*-f)\geq \omega\left(-\frac{\partial p_i}{\partial r}(r^*)\right)\Vert f-f^*\Vert_{L^1(\O)}^2.
\end{equation}
\end{proposition}
\begin{proof}[Proof of Proposition \ref{Pr:BathtubTD}]
Let us write $\mathscr T:=\{p_i\}_{i\in I}$. We first note that the assumption ensure that for any $p\in \mathscr T$, $f^*$ is the only solution of the problem 
\begin{equation}
\sup_{f\in \overline{\mathcal M}(\O)}\int_\O fp.\end{equation}

We define 
\begin{equation}
\mathscr G: \mathscr T\times \left(\overline{\mathcal M}(\O)\backslash \{f^*\}\right) \ni (p,f)\mapsto \frac{ \int_\O p(f^*-f)}{-\frac{\partial p_i}{\partial r}(r^*)\Vert f-f^*\Vert_{L^1(\O)}^2}
\end{equation}
and obviously proving \eqref{Eq:BathtubTD} boils down to proving 
\begin{equation}\inf_{\mathscr T\times \left(\overline{\mathcal M}(\O)\backslash \{f^*\}\right)}\mathscr G>0.\end{equation}
Let us consider a minimising sequence $\{p_k\,, f_k\}\in\left( \mathscr T\times \left(\overline{\mathcal M}(\O)\backslash \{f^*\}\right)\right)^\N$. Let us fix $\beta'\in (0;\beta)$. By \eqref{Eq:HolderUniforme} there exists $p_\infty\in \mathscr C^{1,\beta'}(\O)$ radially symmetric such that \begin{equation}
p_k\underset{k\to \infty}\rightarrow p_\infty\text{ in }\mathscr C^{1,\beta'}(\O)\,, \end{equation} and as a consequence we have 
\begin{equation}\Vert  p_\infty\Vert_{\mathscr C^{1,\beta'}}=\lim_{k\to \infty} \Vert p_k\Vert_{\mathscr C^{1,\beta'}}\leq M\end{equation}
and \eqref{Eq:DerivUniforme} holds for $p_\infty$. In the same way, and passing to the limit in \eqref{Eq:Mack}, $f^*$ is the only solution of 
\begin{equation}\sup_{f \in \overline{\mathcal M}(\O)}\int_\O p_\infty f.\end{equation}

Up to a subsequence we also have that there exists $f_\infty\in \overline{\mathcal M}(\O)$ such that
\begin{equation}f_k\underset{k\to \infty}\rightarrow f_\infty\text{ weakly in $L^\infty-*$}.\end{equation} 

We distinguish between two cases related to the sequence $\{\delta_k\}_{k\in\N}$ defined by 
\begin{equation}
\forall k \in \N\,, \delta_k:=\Vert f_k-f^*\Vert_{L^1(\O)}.\end{equation}  The first case corresponds to the case where, up to a subsequence, 
\begin{equation}
\delta_k\underset{k\to \infty}\rightarrow \delta_\infty>0.
\end{equation}
In that case, we define 
\begin{equation}
\overline{\mathcal M}_{>\delta_\infty}(\O):=\left\{f\in \overline{\mathcal M}(\O)\,, \Vert f-f^*\Vert_{L^1(\O)}\geq \frac{\delta_\infty}2\right\}
\end{equation}
Following the same arguments as in \cite[Proposition 22]{MRB2020} we can see that the class $\overline{\mathcal M}_{>\delta_\infty}(\O)$ is closed under the weak $L^\infty-*$ convergence. Hence, it follows that 
\begin{equation}\Vert f_\infty-f^*\Vert_{L^1(\O)}\geq \frac{\delta_\infty}2.\end{equation} This implies that 
\begin{equation}\lim_{k\to \infty}\mathscr G(p_k,f_k)\geq \frac4{\delta_\infty^2}\mathscr G(p_\infty,f_\infty)>0\end{equation} since $f^*$ is the only maximiser of $f\mapsto \int_\O f p_\infty$ in $\overline{\mathcal M}(\O)$.

The second case is the difficult one. We henceforth work under the assumption that 
\begin{equation}
\delta_k \underset{k\to \infty}\rightarrow 0.\end{equation}

We introduce the sequence of variational problem
\begin{equation}
\forall k \in \N\,, \sup_{f\in \overline{\mathcal M}(\O)\,, \Vert f-f^*\Vert_{L^1(\O)}=\delta_k}\int_{\O} p_kf.
\end{equation}
From the same arguments as in \cite[Proposition 22]{MRB2020} there exists a solution to this variational problem. Furthermore since $p_k^\#=p_k$ the function $\mathds 1_{\A_{\delta_k}}$ is a solution of this problem, where $\mathbb A_{\delta_k}$ is defined, in radial coordinates
\begin{equation}
\overline\A_{\delta_k}=\{r<r^*-r_{\delta_k}^-\}\sqcup \{r^*<r<r^*+r_{\delta_k}^+\}\end{equation} and $r_{\delta_k}^-,r_{\delta_k}^+$ are the unique parameters such that 
\begin{equation}\operatorname{Vol}(\mathbb A_{\delta_k})=V_0\,, \operatorname{Vol}(\mathbb A_{\delta_k}\Delta \B^*)=\delta_k.\end{equation}

Hence we assume that 
\begin{equation}f_k=\mathds 1_{\mathbb A_{\delta_k}}.\end{equation}

For a general $\overline \delta>0$, we define $\mathbb A_{\overline \delta}$ in the same manner, that is, 
\begin{equation}
\A_{\overline\delta}=\{r<r^*-r_{\overline\delta}^-\}\sqcup \{r^*<r<r^*+r_{\overline\delta}^+\}\end{equation} where $r_{\overline\delta}^-,r_{\overline\delta}^+$ are the unique parameters such that 

\begin{equation}\operatorname{Vol}(\A_{\overline \delta})=V_0\,, \operatorname{Vol}\left(\A_{\overline \delta}\Delta \B^*\right)=\overline \delta.\end{equation}

 We also recall that we have, for the same exponent $\beta'>0$, 
\begin{equation}\label{Eq:HolderFaible}
\forall i\in I\,, \Vert  p_i\Vert_{\mathscr C^{1,\beta'}}\leq M\end{equation} as this will be a crucial point. Let us then prove the following claim:
\begin{claim}\label{Cl:Monster}
There exists $\delta_1>0$ and $\omega_0>0$ such that for any $0\leq \delta\leq \delta_1$ there holds
\begin{equation}
\forall p\in \mathscr T\cup \{p_\infty\}\,, 
\int_\O p(f^*-\mathds 1_{\A_\delta})\geq \omega_0\left(-\frac{\partial p}{\partial r}(r^*)\right)\Vert \mathds 1_{\A_\delta}-f^*\Vert_{L^1(\O)}^2.
\end{equation}
\end{claim}
Assuming this Claim holds it follows that for any $k$ large enough we have 
\begin{equation}\mathscr G(p_k,f_k)\geq \omega_0,\end{equation} hence leading to the required contradiction. It thus only remains to prove Claim \ref{Cl:Monster}:
\begin{proof}[Proof of Claim \ref{Cl:Monster}]
Let us define, for any $\delta>0$, 
\begin{equation}
h_\delta:=f^*-\mathds 1_{\A_\delta}=\mathds 1_{\{r^*-r_\delta^-<r<r^*\}}-\mathds 1_{\{r^*<r<r^*+r_\delta^+\}}.
\end{equation}
The quantity we want to bound from below is 
\begin{equation}\int_\O h_\delta p_i.\end{equation}
First of all, explicit computations show that there exists a constant $c_0=c_0(d,r^*)$ such that 

\begin{equation}\label{Eq:Rad}
r_\delta^+\,, r_\delta^-\underset{\delta \to 0}\sim c_0 \delta.
\end{equation}
We now write \eqref{Eq:Rad} in radial coordinates and obtain for any $p\in \mathscr T\cup\{p_\infty\}$
\begin{align*}
\frac1{(2\pi)^d}\int_\O h_\delta p=\int_{r^*-r_\delta^-}^{r^*}p(r)r^{n-1}{d}r-\int_{r^*}^{r^*+r_\delta^+}p(r)r^{n-1}{d}r.
\end{align*}

Let us first notice that from \eqref{Eq:HolderFaible} and \eqref{Eq:DerivUniforme}, there exists $\overline \e>0$ such that, for any $\delta\in (0;\overline \e)$, 
\begin{equation}\label{Eq:Burr}
\forall p\in \mathscr T\cup\{p_\infty\}\,, \inf_{(r^*-\delta;r^*+\delta)}\left(-\frac{\partial p}{\partial r}\right)\geq -\frac12\frac{\partial p}{\partial r}(r^*).
\end{equation}
We now have thanks to the mean value theorem, that for any $r\in (0;R)$, there exists $y_{r-r^*}\in (r;r^*)$ or $(r^*;r)$ such that
\begin{equation} p(r)=p(r^*)+p'(y_{r-r^*})(r-r^*).\end{equation}
As a consequence
\begin{align}
\frac1{(2\pi)^n}\int_\O h_\delta p&=p(r^*)\left(\int_{r^*-r_\delta^-}^{r^*}r^{n-1}{d}r-\int_{r^*}^{r^*+r_\delta^+}r^{n-1}{d}r\right)
\\&+\left(\int_{r^*-r_\delta^-}^{r^*}r^{n-1}|r-r^*|(-p'(y_{r-r^*})){d}r+\int_{r^*}^{r^*+r_\delta^+}r^{n-1}|r-r^*|(-p'(y_{r-r^*})){d}r\right)
\\&\label{I1}\geq p(r^*)\left(\int_{r^*-r_\delta^-}^{r^*}r^{n-1}{d}r-\int_{r^*}^{r^*+r_\delta^+}r^{n-1}{d}r\right)
\\&\label{I2}-\frac12\frac{\partial p}{\partial r}(r^*)\left(\int_{r^*-r_\delta^-}^{r^*}r^{n-1}|r-r^*|{d}r+\int_{r^*}^{r^*+r_\delta^+}r^{n-1}|r-r^*|{d}r\right)\text{ by \eqref{Eq:Burr}}
\end{align}The right hand side of \eqref{I1} is 0 because 
\begin{equation}\int_{r^*-r_\delta^-}^{r^*}r^{n-1}{d}r-\int_{r^*}^{r^*+r_\delta^+}r^{n-1}{d}r=\frac1{(2\pi)^n}\int_\O h_\delta=0.\end{equation} Furthermore by explicit computations we obtain 

\begin{align*}
\int_{r^*-r_\delta^-}^{r^*}r^{n-1}|r-r^*|dr\underset{\delta \to 0}\sim C\delta^2,
\end{align*}
and in the same way
\begin{equation}
\int_{r^*}^{r^*+r_\delta^+}r^{n-1}|r-r^*|dr\underset{\delta \to 0}\sim C\delta^2.
\end{equation}
The conclusion follows immediately.
\end{proof}

This concludes the proof of the Proposition.
\end{proof}

We then present, in the following paragraph, the proof of the aforementioned Proposition \ref{Pr:GoodGuys} that deals with the characterisation of solutions of a penalised problem.

\subsection{Characterisation of the solutions of an auxiliary problem}
Let us consider a function $\delta:[0;T]\rightarrow [0;\operatorname{Vol}(\O)]$ and the class $\mathcal M_T(\O,\delta)$ defined in \eqref{Eq:MTDelta}, as well as the function $f_\delta$ defined by \eqref{Eq:Fdelta}.
\begin{proposition}\label{Pr:GoodGuys}
For any $\e> 0$ and any positive function $\delta:[0;T]\to [0;\operatorname{Vol}(\O)]$, $f_\delta$ is a solution of the variational problem
\begin{equation}
\max_{g\in \mathcal M_T(\O,\delta)}\mathcal J_T^\e(g).
\end{equation}
\end{proposition}
\begin{proof}[Proof of Proposition \ref{Pr:GoodGuys}]
This is a straightforward adaptation of the proof of the parabolic isoperimetric inequality whose main steps were recalled in Section \ref{Se:Uniq}.  Let us consider a function $g\in \mathcal M_T(\O,\delta)$ and $u$ the associated solution of \eqref{Eq:Main}.  With the same notations as in Section \ref{Se:Uniq}, proof of Theorem \ref{Th:Uniqueness} we obtain
\begin{equation}
S_n^2\mu(t,\tau)^{2-\frac2n}\leq \left(-\frac{\partial \mu}{\partial \tau}\right)\int_{\{u(t,\cdot)>\tau\}} \left(g-\frac{\partial u}{\partial t}\right).\end{equation} 
However, by the Hardy-Littlewood inequality, if we define $G(t,\mu(t,\tau)):=\int_{0}^{\mu(t,\tau)} \mathds 1_{\A_{\delta(t)}}^\dagger$ we obtain 
\begin{equation}\int_{\{u(t,\cdot)>\tau\}}g\leq G(t,\mu(t,\tau)).\end{equation}
This is a penalised version of the Hardy-Littlewood inequality: it is indeed straightforward to see that, for any function $\overline g \in \overline{\mathcal M}(\O,\delta(t))=\{g\in \overline{\mathcal M}(\O)\,, \Vert g-f^*\Vert_{L^1(\O)}=\delta(t)\}$ and any measurable positive function $\ell$, there holds
\begin{equation}\int_\O \ell \overline g\leq \int_\O \mathds 1_{\mathbb A_{\delta(t)}}\ell^\#.\end{equation}As a consequence, for some constant $c_n>0$,

\begin{equation}
1\leq S_n^{-2}\left({-\frac{\partial \mu}{\partial \tau}}\right)\mu(t,\tau)^{\frac{2}n-2}\left(G(t,\mu(t,\tau))-\frac{\partial k}{\partial t}(t,\mu(t,\tau))\right).
\end{equation}
The rest of the proof follows along exactly the same lines.

\end{proof}

\subsection{Proof of Theorem \ref{Th:Td}}
In this subsection, we prove Theorem \ref{Th:Td} with a fixed, positive parameter $\e>0$. 
\begin{proof}[Proof of Theorem \ref{Th:Td}]
We argue by contradiction and assume that there exists a sequence $\{f_k\}_{k\in \N}\in(\mathcal M_T(\O)\backslash\{f^*\})^\N$ such that 
\begin{equation}\label{Eq:Contrad}\lim_{k\to \infty}\frac{\mathcal J_T^\e(f^*)-\mathcal J_T^\e(f_k)}{\int_0^T\left(-\frac{\partial p_\e^*}{\partial r}(t,r^*)\right)\Vert f_k(t,\cdot)-f^*\Vert_{L^1(\O)}^2}=0,\end{equation} where we recall that $p_\e^*$ is the solution of 
\begin{equation}\label{Eq:Adjointepsilon}
\begin{cases}
\frac{\partial p_\e^*}{\partial t}+\Delta p_\e^*=-u^*\text{ in }\OT\,, 
\\ p_\e^*(T,\cdot)=\e u^*(T,\cdot)\,, 
\\ p_\e^*(t,\cdot)=0\text{ on } (0;T)\times \partial \O.\end{cases}\end{equation}

In the same way, if $f\in \mathcal M_T(\O)$, $p_{\e,f}$ stands for the solution of \eqref{Eq:Adjointepsilon} with $u^*$ replaced by $u_{f}$. By Proposition \ref{Pr:Adjoint}, the derivative of $\mathcal J_T^\e$ at $f$ in a direction $h$ is given by 
\begin{equation}\label{Eq:DerivAdjointEpsilon}
\dot{\mathcal J}_T^\e(f)[h]=\iint_\OT h p_{\e,f}.\end{equation}

Let us then begin with the following Claim:

\begin{claim}\label{Cl:NDA}
For any $T,\e>0$ and any $y_0\in (0;r^*)$, there exists $\alpha(y_0;\e,T)$ such that 
\begin{equation}
\inf_{(0;T)\times(y_0;R)}\left(-\frac{\partial p_\e^*}{\partial r}(t,r)\right)\geq \alpha(y_0;\e,T)>0.\end{equation}\end{claim}

\begin{proof}[Proof of Claim \ref{Cl:NDA}]
We define $q_\e^*(t,\cdot):=p_\e^*(T-t,\cdot)$. Since $u^*$ is radially symmetric, $q_\e$ is radially symmetric as well and satisfies, in radial coordinates,
\begin{equation}\label{Eq:Qespilon}
\begin{cases}
\frac{\partial q_\e^*}{\partial t}-\frac1{r^{n-1}}\frac{\partial}{\partial r}\left(r^{n-1}\frac{\partial q_\e^*}{\partial r}\right)=u^*(T-t,\cdot)\text{ in } (0;T)\times (0;R)\,, 
\\ q_\e^*(0,\cdot)=\e u^*(T,\cdot)\,, 
\\ q_\e^*(t,R)=\frac{\partial q_\e^*}{\partial r}(t,0)=0.
\end{cases}
\end{equation}

It follows from standard Schauder estimates \cite[Theorem 4.9, p.59]{Lieberman} and Proposition \ref{Pr:Regularity} that $q_\e^*\in \mathscr C^{1,\alpha}(\OT)$. Besides, since $u^*\geq 0$, we also have $q_\e\geq 0$ and, by the strong parabolic maximum principle, $q_\e^*>0\in \OT$.

 As a consequence of the Hopf Lemma and of the fact that $\frac{\partial u^*}{\partial r}(T,R)<0$, defining $\Phi_\e^*:=\frac{\partial q_\e^*}{\partial r}$, we obtain
 \begin{equation}
 \forall t \in [0;T]\, , \Phi_\e^*(t,R)<0.\end{equation}
 From Proposition \ref{Pr:Radial}, $\Phi_\e^*(0,\cdot)\leq 0\,, <0$ in $(0;R]$. Differentiating \eqref{Eq:Qespilon}, $\Phi_\e^*$ thus solves
 \begin{equation}\label{Eq:PhiEpsilon}
 \begin{cases}
\frac{\partial \Phi_\e^*}{\partial t}-\frac1{r{^{n-1}}}\frac{\partial}{\partial r}\left(r^{n-1}\frac{\partial \Phi_\e^*}{\partial r}\right)=\frac{\partial u^*(T-t,\cdot)}{\partial r}-\frac{(n-1)\Phi_\e^*}{r^2}\text{ in } (0;T)\times (0;R)\,, 
\\ \Phi_\e^*(0,r)<0\text{ if }r>0\,, \Phi_\e^*(t,0)=0\,, 
\\ \Phi_\e^*(t,R)<0.
 \end{cases}
 \end{equation}
Since (Proposition \ref{Pr:Radial}) $\frac{\partial u^*}{\partial r}<0$ for almost every $t,r>0$, $\Phi_\e^*$ solves, in $\OT$, the differential inequality
\begin{equation}
\frac{\partial \Phi_\e^*}{\partial t}-\frac1{r^{n-1}}\frac{\partial}{\partial r}\left(r^{n-1}\frac{\partial \Phi_\e^*}{\partial r}\right)<-\frac{(n-1)\Phi_\e^*}{r^2}\text{ in } (0;T)\times (0;R).
\end{equation} We can then apply the maximum principle, as was done in Proposition \ref{Pr:Radial}, to ensure that for any $t\in (0;T]$ and any $r>y_0$, 
\begin{equation}
\Phi_\e^*(t,r)<0.
\end{equation}
As $\Phi_\e^*(0,r^*)=\e\frac{\partial u_\e^*}{\partial r}(T,r^*)<0$ it follows that 
\begin{equation}\forall t\in [0;T]\,, \Phi_\e^*(t,r^*)<0.\end{equation} Since $\Phi_\e$ is continuous in time, we can define 
\begin{equation}\alpha(y_0;\e,T):=\inf_{t\in [0;T], r\in (y_0;R)}\left(-\Phi_\e^*(t,r)\right)>0\end{equation} and the conclusion follows.
\end{proof}

Using this Claim we can come back to the sequence $\{f_k\}_{k\in \N}\in(\mathcal M_T(\O)\backslash\{f^*\})^\N$ satisfying \eqref{Eq:Contrad}.  Since $f^*$ is the unique maximiser of $\mathcal J_T^\e$, we must have 
\begin{equation}\label{Eq:Raise}
\int_0^T\left(-\frac{\partial p_\e^*}{\partial r}(t,r^*)\right)\Vert f_k-f^*\Vert_{L^1(\O)}^2\underset{k\to \infty}\rightarrow 0.
\end{equation}

If this were not the case, it would follow that the sequence $\{f_k\}_{k\in \N}$ converges weakly in $\mathcal M_T(\O)$ to some $f_\infty \neq f^*$. As a consequence, the sequence $\{u_{f_k}\}_{k\in \N}$ would converge in $\mathscr C^0(\OT)$ (using the uniform H\"{o}lder bounds from Proposition \ref{Pr:Regularity}) to $u_{f_\infty}$, and so 
\begin{equation}\mathcal J_T^\e(f_k)\underset{k\to \infty}\rightarrow \mathcal J_T^\e(f_\infty)>\mathcal J_T^\e(f^*).\end{equation} This would yield\begin{equation}\lim_{k\to \infty}\frac{\mathcal J_T^\e(f^*)-\mathcal J_T^\e(f_k)}{\int_0^T\left(-\frac{\partial p_\e^*}{\partial r}(t,r^*)\right)\Vert f_k(t,\cdot)-f^*\Vert_{L^1(\O)}^2}=\frac{\mathcal J_T^\e(f^*)-\mathcal J_T^\e(f_\infty)}{\int_0^T\left(-\frac{\partial p_\e^*}{\partial r}(t,r^*)\right)\Vert f_\infty(t,\cdot)-f^*\Vert_{L^1(\O)}^2}>0,\end{equation} a contradiction.

Hence we work under the assumption that \eqref{Eq:Raise} holds. From Claim \ref{Cl:NDA} this implies 
\begin{equation}\int_0^T\Vert f_k(t,\cdot)-f^*\Vert_{L^1(\O)}^2\underset{k\to \infty}\rightarrow 0.\end{equation} Hence, by Jensen's inequality, 
\begin{equation}\Vert f_k-f^*\Vert_{L^1(\OT)}=\int_0^T\Vert f_k(t,\cdot)-f^*\Vert_{L^1(\O)}dt\underset{k\to \infty}\rightarrow 0.\end{equation}
As a consequence of standard parabolic estimates (Proposition \ref{Pr:Regularity}) we have, for any $\alpha\in (0;1)$, 
\begin{equation}u_{f_k}\underset{k\to \infty}\rightarrow u_{f^*}=u^*\text{ in }\mathscr C^{0,\alpha}(\O)\end{equation} Defining, for any $k\in \N$, $p_\e^{f_k}$ as the solution of 
\begin{equation}\begin{cases}
\frac{\partial p_\e^{f_k}}{\partial t}+\Delta p_\e^{f_k}=-u_{f_k}\text{ in }\OT\,, 
\\ p_\e^{f_k}(T,\cdot)=\e u_{f_k}\text{ in }\O\,, 
\\ p_\e^{f_k}(t,\cdot)=0\text{ in }(0;T)\times \partial \O\,, \end{cases}\end{equation} This in in turn implies, by Schauder's estimates \cite[Theorem 48.2]{Souplet}

\begin{equation}p_\e^{f_k}\underset{k\to \infty}\rightarrow p_\e^*\text{ in }\mathscr C^{1,\alpha}(\OT).
\end{equation}
Hence, for any $y_0\in (0;r^*)$,  there exists $k(y_0)>0$ such that for any $k\geq k(y_0)$, by Claim \ref{Cl:NDA}, there holds, 
\begin{equation}
\forall (t,r) \in (0;T)\times (y_0;R)\,, \left(-\frac{\partial p_\e^{f_k}}{\partial r}(t,r)\right)\geq-\frac12 \frac{\partial p_\e^{*}}{\partial r}(t,r)>0,
\end{equation}
and, for any $k>0$ large enough, $\B(0;r^*)$ is a uniquely defined level set of $p_\e^{f_k}$: there exists $c_k$ such that 
\begin{equation}\B(0;r^*)=\{p_\e^{f_k}>c_k\}.\end{equation}  As a consequence, choosing $y_0$ small enough, we can ensure that all the assumptions of Proposition \ref{Pr:BathtubTD} are satisfied.

Finally, let us note that, by the same argument, these property also hold for any $p_\e^{f^*+\tau(f_k-f^*)}$ for any $\tau \in (0;1)$ and any $k$ large enough. In all the reasoning above, it suffices to add $\tau$ as another parameter in the family.

This allows us to apply Proposition \ref{Pr:BathtubTD}: there exists a constant $\underline \omega>0$ such that
\begin{multline}\label{Almost}
\forall f \in \mathcal M(\O)\,, \forall k\text{ large enough, }\forall t \in (0;T)\,, \forall \tau \in (0;1)\,,\\  \int_\O p_\e^{f^*+\tau(f_k-f^*)}(t,\cdot)(f^*-f_k(t,\cdot))\geq \underline\omega\left(-\frac{\partial p_\e^*}{\partial r}(t,r^*)\right)\Vert f_k(t,\cdot)-f^*\Vert_{L^1(\O)}^2.\end{multline}

Let us now apply, for any $k$ large enough, the mean value theorem to the map 
\begin{equation}T_k=[0;1]\ni \tau\mapsto \mathcal J_T^\e(f^*+\tau(f_k-f^*)).
\end{equation}
There exists $\overline \tau \in (0;1)$ such that
\begin{equation}
\mathcal J_T^\e(f_k-f^*)=\iint_\OT p_\e^{f^*+\overline\tau(f_k-f^*)} \left(f_k-f^*\right).
\end{equation}
Using \eqref{Almost} we get
\begin{equation}
\mathcal J_T^\e(f_k-f^*)=\iint_\OT p_\e^{f^*+\overline \tau(f_k-f^*)} \left(f_k-f^*\right)\geq\underline \omega\int_0^T\left(-\frac{\partial p_\e^*}{\partial r}(t,r^*)\right)\Vert f_k(t,\cdot)-f^*\Vert_{L^1(\O)}^2.
\end{equation}
This is a contradiction, and the Theorem follows.

\end{proof}

\section{Proof of Theorem \ref{Th:Ti}: quantitative inequalities via shape derivatives and bathtub principle}\label{Se:Ti}
\subsection{Presentation and plan of the proof}
The proof relies on the use of shape derivatives and on the study of an auxiliary problem. 
The structure of the proof is inspired by a previous work of the author \cite{MRB2020} and we will refer to this paper when needed. The main point is here to show an example of how shape derivatives may be used for parabolic problems.

Let us define, for any $\delta>0$, the class 
\begin{equation}\tag{$\overline{\bold{Adm}}(\delta)$}\label{Eq:AdmDelta}\overline{\mathcal M}(\O,\delta):=\left\{ f\in \overline{\mathcal M}(\O)\,, \Vert f-f^*\Vert_{L^1(\O)}=\delta\right\}.\end{equation} We first consider the auxiliary variational problem
\begin{equation}\tag{$\bold{P}_\delta$}\inf_{f\in \overline{\mathcal M}(\O,\delta)} \mathcal J_T(f)\end{equation}
and prove that it admits a solution $f_\delta$ (Lemma \ref{Le:ExistenceDelta} below). Once this is done, we prove (Lemma \ref{Le:Intermediaire} below) that Theorem \ref{Th:Ti} is equivalent to proving that
\begin{equation}
\underset{\delta \to 0}{\lim \inf}\left(\frac{\mathcal J_T(f^*)-\mathcal J_T(f_\delta)}{\delta^2}\right)>0.
\end{equation}

\begin{remark}At this stage, one may argue to explicitly characterize $f_\delta$ as a radially symmetric solution, and thus bypass the part about shape derivatives. However, as our goal is also to provide a full analysis of shape hessians for time-dependent problems, and to present, in the Conclusion, possible generalisations to other settings where the explicit characterisation of optimisers of such a penalised problem are no longer available, we choose to not take advantage of that fact here. 
\end{remark}

We then recall that $f^*=\mathds 1_{\B^*}$. We consider, for smooth enough vector fields $\Phi$, the deformed set $\B^*_\Phi:=(Id+\Phi)(\B^*)$ and, with a slight abuse of notation, we write 
$$\mathcal J_T(\B^*_\Phi):=\mathcal J_T(\mathds 1_{\B^*_\Phi}).$$ We will prove (Proposition \ref{Pr:NormalDeformation}) that whenever $\Phi$ is "small" enough (in a sense made precise in the section devoted to shape derivatives) there holds
\begin{equation}\mathcal J_T(\B^*)-\mathcal J_T(\B^*_\Phi)\geq C\operatorname{Vol}(\B^*_\Phi\Delta \B^*)^2\end{equation}
for some constant $C>0$.

We also prove a quantitative bathtub principle (Proposition \ref{Pr:Bathtub}), and finally conclude as in \cite{MRB2020} by comparing any competitor with one of the level sets of the switch function, and then this level set with the set $\B^*$. The key to conclude here is the convexity of the cost functional $\mathcal J_T$.

To proceed, we need some basic informations about the optimality conditions for Problem \eqref{Eq:PvTi}.

\paragraph{Optimality conditions for \eqref{Eq:PvTi}}
We recall, from Proposition \ref{Pr:Adjoint} that for any admissible perturbation $h\in L^\infty(\O)$ (that is, such that, for any $\e>0$ small enough, $f^*+\e h\in \overline{\mathcal M}(\O)$) the G\^ateaux-derivative of $u_f$ in the direction $h$, thereafter noted $\dot u$ solves 
\begin{equation}\label{Eq:Udot}
\begin{cases}
\frac{\partial \dot u}{\partial t}-\Delta \dot u=h\text{ in }\OT\,, 
\\ \dot u=0\text{ on }(\O;T)\times \partial \O\,, 
\\ \dot u(0,\cdot)\equiv 0
\end{cases}
\end{equation}
and that, introducing the solution $p_f$ of\begin{equation}\label{Eq:Adjoint}\begin{cases}
\frac{\partial p_f}{\partial t}+\Delta p_f=-u_f\text{ in } \OT\,, 
\\ p_f=0\text{ on }(0;T)\times \partial \O\,, 
\\ p_f(T,\cdot)\equiv 0.\end{cases}
\end{equation}
we get the following expression for the G\^ateaux-derivative of $\mathcal J_T$:

\begin{equation}\label{Eq:Jdot2}
\dot{\mathcal J_T}(h)=\iint_\OT hp_f=\int_\O h(x)\left(\int_0^Tp_f(t,x)dt\right)dx.
\end{equation}
 Let us define 
\begin{equation}\Psi(x):=\int_0^T p_f(t,x)dt.\end{equation} Hence it follows that  \begin{equation}\dot{\mathcal J_T}(h)=\int_\O h\Psi.\end{equation} 

\subsection{Reduction to an auxiliary problem}
We now justify the study of the auxiliary problem
\begin{equation}\label{Eq:PvDelta}\tag{$\bold{P}_\delta$}\inf_{f\in \overline{\mathcal M}(\O,\delta)} \mathcal J_T(f)\end{equation} where $\overline{\mathcal M}(\O,\delta)$ was defined in \eqref{Eq:AdmDelta}.

\begin{lemma}\label{Le:ExistenceDelta}
For any $\delta>0$, the variational problem \eqref{Eq:PvDelta} has a solution $f_\delta$.
\end{lemma}
This Lemma is an adaptation of \cite[Proposition 22]{MRB2020}; for the sake of readability, its proof is only given in Appendix \ref{Ap:Technical}.
Throughout the rest of the proof of Theorem \ref{Th:Ti} we adopt the following notation:
\begin{equation}\text{ For any $\delta>0$, $f_\delta$ is a solution of \eqref{Eq:PvDelta}.}\end{equation}
We now explain why we will focus on the study of $f_\delta$ as a competitor; it is the subject of the following Lemma:

\begin{lemma}\label{Le:Intermediaire} 
Theorem \ref{Th:Ti} is equivalent to proving that
\begin{equation}\label{Eq:Le}\underset{\delta \to0}{\lim\inf} \frac{\mathcal J_T(f^*)-\mathcal J_T(f_\delta)}{\delta^2}\geq C_0>0\end{equation} for some positive constant $C_0$.
\end{lemma}
The proof of this result is an adaptation of \cite[Lemma 23]{MRB2020} and mostly relies on the uniqueness of maximisers. We postpone the proof to Appendix \ref{Ap:Technical}.
The rest of the proof of Theorem \ref{Th:Ti} is going to be devoted to the proof of Estimate \eqref{Eq:Le}, see Proposition \ref{Pr:Le}  below. To prove it, we need a local inequality for deformations of the optimal set $\B^*$ and a quantitative bathtub principle which will be used in combination with the convexity of the functional.

\subsection{Quantitative inequalities for deformations of $\B^*$: using shape derivatives} 
Let us consider a $\mathscr C^1$ set $E$ of volume $V_0$ such that $E\cap \partial \O=\emptyset$ and a smooth, compactly supported in $\O$, vector field $\Phi$. We define
\begin{equation}E_\Phi:=(Id+\Phi)(E).\end{equation} We recall that we see $\mathcal J_T$ as a shape functional by defining, with a slight abuse of notations, 
\begin{equation}\mathcal J_T(E):=\mathcal J_T(\mathds 1_E).\end{equation}Our goal is the following proposition:

\begin{proposition}\label{Pr:NormalDeformation}
There exist a constant $C>0$, a parameter $\eta>0$ and $p\in (1;+\infty)$ such that, for any compactly supported vector field $\Phi$ satisfying $\Vert \Phi\Vert_{W^{2,p}}$ there holds
\begin{equation}
\mathcal J_T(\B^*)-\mathcal J_T(\B^*_\Phi)\geq C\operatorname{Vol} \left(\B^*_\Phi\Delta \B^*\right)^2.
\end{equation}
\end{proposition}

The proof of this Proposition follows the synthetic presentation of quantitative inequalities for deformations of optimal sets presented in \cite{DambrineLamboley}; their proof holds for shape optimisation of the domain $\O$, and we have presented in \cite{MRB2020} how to adapt their method to the optimisation of a subdomain $E\subset \O$. Let us present the main steps of the proof of Proposition \ref{Pr:NormalDeformation}:
\begin{enumerate}
\item The first one is to prove that $\B^*$ is a critical shape in the following sense: computing, for any compactly supported vector field $\Phi\in W^{2,p}$ the first order shape derivative $\mathcal J_T'(\B^*)[\Phi]$ we need to prove that, if $\Phi$ additionally satisfies the linearised constraint 
\begin{equation}\label{Eq:Linearvol}\int_{\partial \B^*} \left(\Phi\cdot \nu \right)=0\end{equation} then there holds 
\begin{equation}\mathcal J_T'(E^*)[\Phi]=0.\end{equation}
This allows to consider, for the computation and analysis of second-order shape derivatives, vector fields $\Phi$ that are normal to $\partial \B^*$, and also to define a Lagrangian associated with a Lagrange multiplier
\begin{equation}
\mathcal L_T(E):=\mathcal J_T(E)+\mu\operatorname{Vol}(E),\end{equation} which satisfies, for any compactly supported vector field $\Phi\in W^{2,p}$ not necessarily satisfying \eqref{Eq:Linearvol}
\begin{equation} \mathcal L_T'(\B^*)[\Phi]=0.\end{equation}
\item As a second step, we compute the second order shape derivative of the Lagrangian $\mathcal L_T''(\B^*)[\Phi,\Phi]$ and prove an $L^2$ coercivity estimate, i.e that there exists a constant $c_0>0$ such that
\begin{equation}\forall \Phi \in W^{2,p}(\O;\R^2)\,, \int_{\partial \B^*}\phi\cdot \nu=0\Rightarrow \mathcal L_T''(\B^*)[\Phi,\Phi]\leq- c_0 \Vert \Phi \cdot \nu \Vert_{L^2(\partial \B^*)}^2.\end{equation}
This is done using a comparison principle previously used for elliptic equations \cite{MazariNadinPrivat,MazariQuantitative}, and our contribution here is to show how it extends to the case of parabolic equations.
\item We then define for a compactly supported vector field $\Phi \in W^{2,p}$ the map \begin{equation}j_\Phi:[0;1]\ni t\mapsto \mathcal L_T(\B^*_{t\Phi})+C(\operatorname{Vol}(\B_{t\Phi}^*)-V_0)^2\end{equation} for some $C$ large enough such that 
\begin{equation}j_\Phi''(0)\leq -\tilde c_0\Vert \Phi \cdot \nu \Vert_{L^2(\partial \B^*)}^2\end{equation} and prove that there exists a modulus of continuity $\eta$, that is, a continuous function $\eta:\R_+\to \R_+$ such that $\omega(0)=0$, such that 
\begin{equation}|j_\Phi''(t)-j_\Phi''(0)|\leq \eta\left(\Vert \Phi\Vert_{W^{2,p}}\right) \Vert \Phi \cdot \nu \Vert_{L^2(\partial \B^*)}^2,\end{equation} and conclude using the Taylor-Lagrange formula
\begin{equation}j_\Phi(1)-j_\Phi(0)=\int_0^1 j_\Phi''(s)ds\leq \left(-\tilde c_0+\omega\left(\Vert \Phi\Vert_{W^{2,p}}\right)\right)\Vert \Phi \cdot \nu \Vert_{L^2(\partial \B^*)}^2.\end{equation} 
\end{enumerate}
All these steps rely on fine properties of first and second order shape derivatives. We begin with the computations of the shape derivatives of the Lagrange multiplier associated with the volume constraint and of the diagonalisation of the associated shape hessian at $E^*$.

\subsubsection{Computation of first and second order shape derivatives, computation of the Lagrange multiplier and diagonalisation of the shape Hessian}
\paragraph{Computation and analysis of the first order shape derivative}
Let us define, for any  subdomain $E$ of $\O$ the function $u_E$ as the solution of \eqref{Eq:Main} associated with $f=\mathds 1_E.$ It should be noted that the shape differentiability of first and second order of the shape functional $\mathcal J_T:E\mapsto \mathcal J_T(E)$ follows from the same arguments as in \cite{Harbrecht}, and so does the computation of the first order shape derivative. The computations are a straightforward adaptation of \cite{Harbrecht} and we only give here a heuristic approach. Let us, then, consider a $\mathscr C^1$ shape, and a $W^{2,p}$ compactly supported vector field $E$. The shape derivative of $E\mapsto u_E$ in the direction $\Phi$ is denoted by $u'$ for the sake of notational simplicity. The differentiation of the main equation of \eqref{Eq:Main}  gives, in a weak form, that, for any test function $v$,
\begin{equation}-\iint_\OT \frac{\partial v}{\partial t}u'+\iint_\OT\langle \n u',\n v\rangle=\iint_{(0;T)\times\partial E}v\left(\Phi \cdot \nu\right).\end{equation} \begin{remark} Alternatively, at a formal level: the differentiation of the initial condition yields 
\begin{equation} u'(0,\cdot)\equiv 0.\end{equation} The differentiation of the main equation gives 
\begin{equation}\frac{\partial u'}{\partial t}-\Delta u'=0.\end{equation}
Finally, the structural condition given by the weak formulation of \eqref{Eq:Main} is that there is no jump of the normal derivative on $\partial \B^*$ or, mathematically, that
\begin{equation}\label{Eq:NiJump}\left.\left\llbracket \frac{\partial u_E}{\partial \nu}\right\rrbracket\right|_{\partial E}=0.\end{equation}
We refer to Subsection \ref{Su:Notation} for the definition of the jump. Differentiating \eqref{Eq:NiJump} yields 
\begin{equation}
 \left.\left\llbracket \frac{\partial u'}{\partial \nu}\right\rrbracket\right|_{\partial E}=-\left(\Phi\cdot \nu\right).
\end{equation}
\end{remark}
In conclusion, $u'$ satisfies 
\begin{equation}\label{Eq:Uprime}
\begin{cases}
\frac{\partial u'}{\partial t}-\Delta u'=0\text{ in }\OT\,, 
\\ u'(t,\cdot)=0\text{ on }(0;T)\times \partial \O\,, 
\\ \left.\left\llbracket \frac{\partial u'}{\partial \nu}\right\rrbracket\right|_{\partial E}=-\left(\Phi\cdot \nu\right),
\\ u'(0,\cdot)\equiv0.\end{cases}\end{equation}

Furthermore, if we consider the adjoint state $p_{\mathds 1_E}$, which we abbreviate as $p_E$ for notational simplicity, given by Equation \eqref{Eq:Adjoint} we obtain 
\begin{align*}
\mathcal J_T'(E)[\Phi]&=\iint_\OT u_Eu'
\\&=-\iint_\OT \left(\frac{\partial p_E}{\partial t}+\Delta p_E\right)u'
\\&=-\iint_{(0;T)\times \partial E}\left\llbracket\frac{\partial u'}{\partial\nu}\right\rrbracket p_E
\\&=\iint_{(0;T)\times \partial E} \left(\Phi\cdot \nu\right)p_E.
\end{align*}

Let us single out this last identity:
\begin{equation}\label{Eq:JPrime}
\mathcal J_T'(E)[\Phi]=\int_{\partial E}\left(\Phi\cdot \nu\right)\left(\int_{0}^T p_E\right).\end{equation}
This allows us to obtain the following result:
\begin{lemma}\label{Le:SCritical}
$\B^*$ is a critical shape in the following sense: for any compactly supported vector field $\Phi \in W^{2,p}(\O;\R^2)$
\begin{equation}
\int_{\partial \B^*}\left(\Phi\cdot\nu\right)=0\Rightarrow \mathcal J_T(\B^*)[\Phi]=0.\end{equation}\end{lemma}

\begin{proof}[Proof of Lemma \ref{Le:SCritical}]
From Proposition \ref{Pr:Radial}, $u^*$ is a radially symmetric function. Hence, the associated adjoint state $p^*=p_{\mathbb B^*}$ is also radially symmetric, so that the map 
\begin{equation}\Psi:\B(0;R)\ni x\mapsto \int_{0}^T p^*(t,x)dt\end{equation} is radially symmetric. Letting $\overline \Psi_{\partial \B^*}:=\left.\Psi\right|_{\partial \B^*}$ we obtain 
\begin{equation}\mathcal J_T'(\B^*)[\Phi]=\overline \Psi_{\partial \B^*}\int_{\partial \B^*}\left( \Phi\cdot \nu\right)=0.\end{equation}
\end{proof}

It follows that the Lagrange multiplier associated with the volume constraint is $\mu=-\overline{\Psi}_{\partial \B^*}$ and we can hence define the Lagrangian
\begin{equation}
\mathcal L_{\B^*}(E):=\mathcal J_T(E)-\overline{\Psi}_{\partial \B^*}\operatorname{Vol}(E)\end{equation} and observe that, since $\operatorname{Vol}'(E)[\Phi]=\int_{\partial E} \left(\Phi\cdot \nu\right)$ we have, for any compactly supported vector field $\Phi\in W^{2,p}(\O;\R^2)$
\begin{equation}
\mathcal L_{\B^*}'(\B^*)[\Phi]=0.\end{equation}
As a consequence of \cite[Theorem~5.9.2 and the remark below]{HenrotPierre}, the second-order shape derivative in a direction $\Phi$ only depends on the normal trace of $\Phi$ and we hence work under the Assumption:
\begin{equation}\tag{$\bold{A}_\nu$}\label{As:Normal}\text{$\Phi$ is normal to $\partial \B^*$.}\end{equation}

\paragraph{Computation of the shape hessian and diagonalisation of the shape hessian at the ball}
We can now turn to the computation of the second order shape derivative. We once again choose a $\mathscr C^2$ shape $E$ and a compactly supported vector field $\Phi\in W^{2,p}(\O;\R^2)$. It is well-known \cite[Proposition 5.4.18]{HenrotPierre} that 
\begin{equation}\operatorname{Vol}''(E)[\Phi,\Phi]=\int_{\partial E} \mathscr H\left(\Phi\cdot \nu\right)^2,\end{equation} where $\mathscr H$ is the mean curvature of $\partial E$. Furthermore, differentiating \eqref{Eq:JPrime} and using once again \cite[Proposition 5.4.18]{HenrotPierre} we obtain
\begin{equation}\mathcal J_T''(E)[\Phi,\Phi]=\int_{\partial E}\left(\Phi\cdot \nu\right)\left(\int_{(0;T)} p'\right)+\int_{\partial E}\left(\Phi\cdot \nu\right)^2\left(\mathscr H \int_{0}^T p_E+\int_0^T\frac{\partial p_E}{\partial \nu}\right)\end{equation} where $p'$ satisfies 
\begin{equation}\label{Eq:AdjointPrime}\begin{cases}
\frac{\partial p'}{\partial t}+\Delta p'=-u'\text{ in }\OT,
\\ p'=0\text{ on }(0;T)\times \O\,, 
\\ p'(T,\cdot)\equiv 0\text{ in }\Omega.\end{cases}\end{equation}

In particular, the shape hessian of the Lagrangian at the ball is given by
\begin{align*}
\mathcal L_{\B^*}''(\B^*)[\Phi,\Phi]&=\int_{\partial \B^*}\left(\Phi\cdot \nu\right)\left(\int_{(0;T)} p'\right)+\int_{\partial \B^*}\left(\Phi\cdot \nu\right)^2\left(\mathscr H\left.\overline{\Psi}\right|_{\partial \B^*}+\int_0^T\frac{\partial p^*}{\partial \nu}\right)-\left.\overline{\Psi}\right|_{\partial \B^*}\int_{\partial \B^*}\mathscr H \left(\Phi\cdot \nu\right)^2
\end{align*}
so that simplifying the terms involving the mean curvature we are left with 
\begin{equation}\label{Eq:ShapeHessianBall}
\mathcal L_{\B^*}''(\B^*)[\Phi,\Phi]=\int_{\partial \B^*}\left(\Phi\cdot \nu\right)\left(\int_{(0;T)} p'\right)+\int_{\partial \B^*}\left(\Phi\cdot \nu\right)^2\int_0^T\frac{\partial p}{\partial \nu}.\end{equation}

Let us now diagonalise it. Since $\Phi$ is a vector field that is normal to $\mathbb S^*:=\partial \B^*$ from Assumption \eqref{As:Normal} it follows that we can decompose it, in angular coordinates, as 
\begin{equation}\label{Eq:Decomposition}
\Phi\cdot \nu=\sum_{k=1}^\infty \alpha_k \cos(k\cdot)+\beta_k\sin(k\cdot).\end{equation}
\begin{remark} The fact that the sum involving the cosines starts at $k=1$ is a consequence of the fact that to compute the optimality condition for second order shape derivative we need to work in the space satisfying the linearised constraint or, in this case, to assume that 
\begin{equation}\int_{\partial \B^*}\Phi\cdot \nu=0.\end{equation}\end{remark}

Let us first define $u_k'$ (resp. $v_k'$) as the solution of \eqref{Eq:Uprime} associated with $\Phi\cdot \nu=\cos(k\cdot)$ (resp. $\sin(k\cdot)$). It is straightforward to see that these two functions write 
\begin{equation}
u_k'(r,\theta)=y_k(r)\cos(k\theta)\,, v_k'(r,\theta)=y_k(r)\sin(k\theta)\end{equation} where $y_k$ solves, for any $k\in \N^*$, 

\begin{equation}\begin{cases}
\frac{\partial y_k}{\partial t}-\frac1r\frac{\partial }{\partial r}(r\frac{\partial y_k}{\partial r})=-\frac{k^2}{r^2}y_k\text{ in }(0;R)\,, 
\\ \left\llbracket y_k'\right\rrbracket(r^*)=-1\,, 
\\ y_k(R,\cdot)=0\,, 
\\ y_k'(0)=0.\end{cases}\end{equation}

Let us also introduce $g_k'$ (resp. $w_k'$) the solution of \eqref{Eq:AdjointPrime} associated with $\Phi\cdot \nu=\cos(k\cdot)$ (resp. $\sin(k\cdot)$). It is straightforward to see that these two functions write 
\begin{equation}
g_k'(r,\theta)=z_k(r)\cos(k\theta)\,, w_k'(r,\theta)=z_k(r)\sin(k\theta)\end{equation} where $z_k$ solves, for any $k\in \N^*$, 

\begin{equation}\begin{cases}
\frac{\partial z_k}{\partial t}+\frac1r\frac{\partial}{\partial r}(r\frac{\partial z_k}{\partial r})=\frac{k^2}{r^2}z_k-y_k\text{ in }(0;R)\,, 
\\ z_k(R,\cdot)=0\,, 
\\ z_k(T)=0.\end{cases}\end{equation}

Furthermore, since $p^*$ is a radially symmetric function let us introduce the function $\overline p$ such that 
\begin{equation}p^*(t,r,\theta)=\overline p(t,r).\end{equation}

This allows to recast the second order shape derivative \eqref{Eq:ShapeHessianBall} through the following Lemma:
\begin{lemma}\label{Le:Diagonalisation}
If $\Phi\cdot \nu$ is of the form \eqref{Eq:Decomposition} then there holds 
\begin{equation}
\mathcal L_{\B^*}''(\B^*)[\Phi,\Phi]=\frac{r^*}2\sum_{k=1}^\infty \omega_k \left\{ \alpha_k^2+\beta_k^2\right\} \end{equation} where for every $k\in \N^*$ we have defined 
\begin{equation}
\omega_k:=\int_0^T z_k(t,r^*)dt+\int_0^T \frac{\partial \overline p}{\partial r}(t,r^*)dt.
\end{equation}
\end{lemma}

\begin{proof}[Proof of Lemma \ref{Le:Diagonalisation}]
We can write \eqref{Eq:ShapeHessianBall} as
\begin{align*}
\mathcal L_{\B^*}''(\B^*)[\Phi,\Phi]&=\int_0^{2\pi}\left(\Phi\cdot \nu\right)\left(\int_{(0;T)} p'\right)+\int_0^{2\pi}\left(\Phi\cdot \nu\right)^2\int_0^T\frac{\partial p^*}{\partial \nu}
\\&=\sum_{k,k'=1}^\infty\int_{0}^{2\pi}\int_0^T\left(\alpha_k\alpha_{k'}\cos(k\theta)\cos(k'\theta)+\beta_k\beta_{k'}\sin(k\theta)\sin(k'\theta)\right.
\\&+\left.+\alpha_k\beta_{k'} \cos(k\theta)\sin(k'\theta)\right)z_k(t,r^*)dtd\theta
\\&+\sum_{k=1}^\infty\int_0^{2\pi}\int_0^T \frac12\left(\alpha_k^2+\beta_k^2\right)\frac{\partial \overline p}{\partial r}(t,r^*)dt.
\end{align*}
All the crossed terms disappear for $k\neq k'$, and the conclusion follows by integrating in polar coordinates.
\end{proof}

We may now state the main result of this subsection:
\begin{proposition}\label{Pr:Coercivity}
There exists a constant $c_0>0$ such that for any $\Phi\in W^{2,p}$ satisfying \eqref{As:Normal} there holds
\begin{equation}\mathcal L_{\B^*}''(\B^*)[\Phi,\Phi]\leq -c_0\int_{\partial \B^*} \left(\Phi\cdot\nu\right)^2.
\end{equation}
\end{proposition}
\begin{proof}[Proof of Proposition \ref{Pr:Coercivity}]
Given Lemma \ref{Le:Diagonalisation} it suffices to prove that there exists a constant $c_0>0$ such that 
\begin{equation}\label{Eq:Goal}
\forall k\in \N^*\,, \omega_k\leq -c_0.
\end{equation}
Equation \eqref{Eq:Goal} is obviously provided the following Claim holds:
\begin{claim}\label{Cl:Paty}
The sequence $\{\omega_k\}_{k\in \N^*}$ is decreasing. Furthermore, $\omega_1<0$.
\end{claim}
Indeed, it then suffices to take $c_0=-\frac2{r^*}\omega_1$ and we can then bound 
\begin{equation}\mathcal L_{\B^*}''(\B^*)[\Phi,\Phi]=\frac{r^*}2\sum_{k=1}^\infty \omega_k(\alpha_k^2+\beta_k^2)\leq -c_0\sum_{k=1}^\infty (\alpha_k^2+\beta_k^2)=-c_0\int_{\partial \B^*}\left(\Phi\cdot\nu\right)^2.\end{equation} We now focus on the proof of this last Claim.
\begin{proof}[Proof of Claim \ref{Cl:Paty}]
Let us note that from Lemma \ref{Le:Diagonalisation} we have
\begin{equation}
\forall k \in \N^*\,, \omega_{k}-\omega_1=\int_0^T \left(z_{k}-z_1\right)(t,r^*)dt.
\end{equation}
The fact that $\{\omega_k\}_{k\in \N}$ is decreasing is thus guaranteed provided the following estimate holds:
\begin{equation}\label{Eq:Zk}
\forall k\in \N^*\,, z_{k}\leq z_1.\end{equation}
\eqref{Eq:Zk} will be proved using a comparison principle. If we want to compare $z_k$ and $z_1$, we need to compare, for any $k\in \N^*$, $y_k$ and $y_1$.
 The first thing to observe is that
\begin{equation}\label{Eq:Y1positif}
y_1\geq 0.\end{equation}
\begin{proof}[Proof of \eqref{Eq:Y1positif}]
We already know that $y_1$ satisfies
\begin{equation}
\begin{cases}
\frac{\partial y_1}{\partial t}-\frac1r\frac{\partial }{dr}(r\frac{\partial y_1}{dr})=-\frac{1}{r^2}y_1\text{ in }(0;R)\,, 
\\ \left\llbracket y_1'\right\rrbracket(r^*)=-1\,, 
\\ y_1(R,\cdot)=0\,, 
\\ y_1'(0,\cdot)=0.\end{cases}\end{equation} We consider the negative part $y_1^-$ of $y_1$. We have  $$\llbracket(y_1^-)'\rrbracket(t,r^*)\begin{cases}=0\text{ if $y_1(t,r^*)>0$,}\\ =1\text{ if $y_1(t,r^*)<0$}, \\>0 \text{ if $y_1(t,\cdot)$ locally changes sign at $r^*$}.\end{cases}$$ In any case, we obtain 
\begin{equation}\left\llbracket (y_1^-)'\right\rrbracket\geq 0.\end{equation}
Multiplying the equation by $y_1^-$ and integrating by parts in space and time as in the proof of Proposition \ref{Pr:Radial} gives 
\begin{equation}\frac12\int_\O (y_1^-)^2(T,\cdot)+\iint_\OT |\n y_1^-|^2+\iint_{(0;T)\times\partial \B^*}y_1^-\left\llbracket (y_1^-)'\right\rrbracket+\iint_\OT \frac1{r}(y_1^-)^2=0.\end{equation} As a conclusion, $y_1^-\equiv 0$, which concludes the proof. 

\end{proof}
Using this information, we can now prove:
\begin{equation}\label{Eq:Comparison1K}
\forall k\in \N^*\,,  y_k\leq y_1.\end{equation}

\begin{proof}[Proof of \eqref{Eq:Comparison1K}]
Let us define, for any $k\in \N^*$, 
\begin{equation}\Psi_k:=y_k-y_1.\end{equation}
Then, in $\OT$, $\Psi_k$ solves 
\begin{equation}
\frac{\partial \Psi_k}{\partial t}-\Delta \Psi_k=-\frac{k^2}{r^2}y_k+\frac1{r^2}y_1\leq -\frac{k^2}{r^2} (y_k-y_1)=-\frac{k^2}{r^2}\Psi_k,\end{equation} where the last inequality comes from the fact that $y_1$ is non-negative. Furthermore, 
\begin{equation}
\left\llbracket \Psi_k'\right\rrbracket(t,r^*)=0,
\end{equation}
so that, following exactly the main line of reasoning, we obtain
\begin{equation}\Psi_k\leq 0,\end{equation} which concludes the proof.

\end{proof}

We now pass to the next step:
\begin{equation}\label{Eq:Z1positive}
z_1\geq 0\text{ in }\OT.\end{equation}
\begin{proof}[Proof of \eqref{Eq:Z1positive}]
The function $z_1$ satisfies

\begin{equation}\begin{cases}
\frac{\partial z_1}{\partial t}+\frac1r\frac{\partial }{\partial r}(r\frac{\partial z_1}{\partial r})=\frac{1}{r^2}z_1-y_1\text{ in }(0;T)\times (0;R)\,, 
\\ z_1(R,\cdot)=0\,, 
\\ z_1(T,0)=\partial_r z_1(t,0)=0.\end{cases}\end{equation}
Since $y_1\geq 0$ from \eqref{Eq:Y1positif} $z_1$ solves, in particular,

\begin{equation}\label{Eq:Z1}\begin{cases}
\frac{\partial z_1}{\partial t}+\frac1r\frac{\partial }{\partial r}(r\frac{\partial z_1}{\partial r})\leq\frac{1}{r^2}z_1\text{ in }(0;T)\times (0;R)\,, 
\\ z_1(R,\cdot)=0\,, 
\\ z_1(T,0)=\partial_r z_1(t,0)=0.\end{cases}\end{equation}
Let us now define $\overline z_1=z_1(T-t,\cdot)$. Straightforward computations show that $\overline z_1$ solves \begin{equation}\partial_t \overline z_1-\frac{1}r\partial_r(r\partial_r\overline z_1)\geq -\frac1{r^2}\overline z_1.\end{equation} Multiplying this identity by $\overline z_1^-$ and integrating by parts,  we obtain in the same way
\begin{equation}z_1\geq 0,\end{equation} as claimed.
\end{proof}
We are now in a position to prove \eqref{Eq:Zk}:
\begin{proof}[Proof of \eqref{Eq:Zk}]
We define, for any $k\in \N$, $\mathscr Z_k:=z_k-z_1$. It is clear that $\mathscr Z_k$ solves
\begin{equation}
\frac{\partial \mathscr Z_k}{\partial t}+\frac1r\frac{\partial }{\partial r}\left(r\frac{\partial  \mathscr Z_k}{\partial r}\right)=\frac{k^2}{r^2}z_k-\frac1{r^2}z_1+y_1-y_k.\text{ in }(0;T)\times (0;R)
\end{equation}
From Estimate \eqref{Eq:Comparison1K} there holds
\begin{equation}
\frac{\partial \mathscr Z_k}{\partial t}+\frac1r\frac{\partial }{\partial r}\left(r\frac{\partial  \mathscr Z_k}{\partial r}\right)\geq\frac{k^2}{r^2}z_k-\frac1{r^2}z_1\text{ in }(0;T)\times (0;R)
\end{equation}
and so, from Estimate \eqref{Eq:Z1positive} we get
\begin{equation}
\frac{\partial \mathscr Z_k}{\partial t}+\frac1r\frac{d}{dr}\left(r\frac{d \mathscr Z_k}{dr}\right)\geq\frac{k^2}{r^2}z_k-\frac{k^2}{r^2}z_1=\frac{k^2}{r^2}\mathscr Z_k\text{ in }(0;T)\times (0;R).
\end{equation}
From the same reasoning, we obtain 
\begin{equation}\mathscr Z_k\leq 0\end{equation} and so
\begin{equation}z_k\leq z_1\text{ in }(0;T)\times (0;R).\end{equation}

\end{proof}
The proof of the first part of Claim \ref{Cl:Paty} is thus finished, and it hence remains to prove that 
\begin{equation}\label{Eq:Omega1}
\omega_1<0.\end{equation}
\begin{proof}[Proof of \eqref{Eq:Omega1}]
We recall that $$\omega_1=\int_0^T z_1(t,r^*)dt+\int_0^T \frac{\partial \overline p}{\partial r}(t,r^*)dt.$$ First of all, is is easy to see that $p$ is non-negative.

Let us define $\overline \p:=\frac{\partial \overline p}{\partial r}$. Straightforward computations show that $\overline \p$ solves 
\begin{equation}\label{Eq:Phi}
\begin{cases}
\frac{\partial \op}{\partial t}+\frac1r\frac{\partial }{\partial r}\left(r\frac{\partial \op}{\partial r}\right)=-\frac{\partial \overline u}{\partial r}+\frac1{r^2}\op\text{ in }(0;T)\times (0;R)\,, 
\\ \op(t,R)\leq 0\,, 
\\ \op(T,\cdot)\equiv 0.
\end{cases}\end{equation}

If we define $\overline \Phi:=\op +z_1$ we thus have
\begin{equation}
\frac{\partial \overline \Phi}{\partial t}+\frac1r\frac{\partial }{\partial r}\left(r\frac{\partial \overline \Phi}{\partial r}\right)=\frac1{r^2}\overline \Phi-\frac{\partial \overline u}{\partial r}\geq \frac1{r^2}\overline \Phi.
\end{equation}
The last inequality comes from Proposition \ref{Pr:Radial}. Furthermore we have $\overline \Phi(t,R)\leq 0$. As a consequence, we have 
\begin{equation}\overline \Phi\leq 0\text{ in }(0;T)\times \O.\end{equation} Furthermore, we necessarily have $\overline \Phi(t,r^*)<0$ in a subset of positive measure of $(0;T)$, for otherwise we have $\frac{\partial \overline u}{\partial r}(t,r^*)= 0$ on this subset, which is absurd given Proposition \ref{Pr:Radial}. As a conclusion, we obtain
\begin{equation}\omega_1=\int_0^T\overline \Phi(t,r^*)dr<0,\end{equation} as claimed.
\end{proof}
\end{proof}
\end{proof}

With this Proposition available, we are in a position to prove Proposition \ref{Pr:NormalDeformation}. Let us recall that, for a normal deformation $\Phi\in W^{2,p}$ we have defined
\begin{equation}j_\Phi(\xi):=\mathcal L_{\B^*}(\B^*_{t\Phi})+C(\operatorname{Vol}(\B_{t\Phi}^*)-V_0)^2.\end{equation} Since $\B^*$ is a critical shape we obtain 
\begin{equation}j_\Phi'(0)=0\end{equation} so that the Taylor-Lagrange formula with integral remainder writes, in the case where $\operatorname{Vol}(\B^*_\Phi)=V_0$, 
\begin{equation}\mathcal L_{\B^*}(\B_\Phi^*)-\mathcal L_{\B^*}(\B^*)=\int_0^1j''\leq j''_\Phi(0)+\int_0^1\left|j_\Phi''(\xi)-j_\Phi''(0)\right| d\xi.\end{equation}

The key is now to prove the following Lemma:
\begin{lemma}\label{Le:Control}
There exists a modulus of continuity, that is, a continuous function $\eta:\R_+\to \R_+$ such that $\omega(0)=0$, such that 
\begin{equation}|j_\Phi''(t)-j_\Phi''(0)|\leq \eta\left(\Vert \Phi\Vert_{W^{2,p}}\right) \Vert \Phi \cdot \nu \Vert_{L^2(\partial \B^*)}^2.\end{equation}
\end{lemma}
Indeed, Lemma \ref{Le:Control} implies Proposition \ref{Pr:NormalDeformation} in the following way: assuming it holds then
\begin{align}\mathcal L_{\B^*}(\B_\Phi^*)-\mathcal L_{\B^*}(\B^*)&=\int_0^1j''\leq j''_\Phi(0)+\int_0^1\left|j_\Phi''(\xi)-j_\Phi''(0)\right| d\xi
\\ &\leq -c_0\Vert \Phi\cdot\nu\Vert_{L^2(\partial \B^*)}^2+ \eta\left(\Vert \Phi\Vert_{W^{2,p}}\right) \Vert \Phi \cdot \nu \Vert_{L^2(\partial \B^*)}^2
\\&\leq -\frac{c_0}2\Vert \Phi\cdot\nu\Vert_{L^2(\partial \B^*)}^2 \text{ for $\Vert \Phi\Vert_{W^{2,p}}$ small enough}
\\&\leq -\frac{c_0}2\frac{1}{\operatorname{Per(\partial \B^*)}}\left(\int_{\partial \B^*} |\Phi \cdot \nu|\right)^2 \text{ by the Cauchy-Schwarz inequality}
\\&\leq -\tilde c_0 \operatorname{Vol}\left(\B^*_\Phi\Delta \B^*\right)^2.\end{align}

 The proof of Lemma \ref{Le:Control} is extremely similar to the proof of \cite[Proposition 23]{MRB2020} and is mostly a technical adaptation of \cite{DambrineLamboley}. For this reason, we postpone it to Appendix \ref{Ap:Control} and briefly sketch here why this $L^2$ norm of $\Phi\cdot \nu$ is, in contrast to the $H^{\frac12}$ usually required in shape optimisation \cite{DambrineLamboley}, the optimal norm here. If we consider, for instance, at at given shape $E$ the second order shape derivative of the Lagrangian, we have
 
\begin{equation}\mathcal L_{\B^*}''(E)[\Phi,\Phi]=\underbrace{\iint_{(0;T)\times \partial E}p'\left(\Phi\cdot \nu\right)}_{=I_1}+\underbrace{\int_{\partial E}\left(\Phi\cdot \nu\right)^2\left(\mathscr H \int_{0}^T p_E+\int_0^T\frac{\partial p_E}{\partial \nu}\right)-\overline \Psi_{\partial \B^*}\int_{\partial E}\mathscr H \left(\Phi\cdot \nu\right)^2}_{=:I_2}\end{equation} where:
\begin{enumerate}
\item $\mathscr H$ is the mean curvature of $\partial E$,
\item $p_E$ solves
\begin{equation}\begin{cases}
\frac{\partial p_E}{\partial t}+\Delta p_E=-u_E\text{ in }\OT\,, 
\\ p_E(T,\cdot)=0\,, 
\\ p_E(t,\cdot)=0\text{ on }(0;T)\times \partial \O,
\end{cases}
\end{equation}

\item $u'$ solves 
\begin{equation}\begin{cases}
\frac{\partial u'}{\partial t}-\Delta u'=0\text{ in }\OT\,, 
\\ u'(0,\cdot)=0\,, 
\\\llbracket \partial_\nu u'\rrbracket=-1\text{ on }(0;T)\times \partial E,
\\ u'(t,\cdot)=0\text{ on }(0;T)\times \partial \O,
\end{cases}
\end{equation}
\item  $p'$ solves

\begin{equation}\begin{cases}
\frac{\partial p'}{\partial t}+\Delta p'=-u'\text{ in }\OT\,, 
\\ p'(T,\cdot)=0\,, 
\\ p'(t,\cdot)=0\text{ on }(0;T)\times \partial \O,
\end{cases}
\end{equation}
\item and 
$$\left.\overline \Psi_{\partial \B^*}=\int_0^T p^*(t,\cdot)\right|_{\partial \B^*}$$ is the Lagrange multiplier associated with the volume constraint.
\end{enumerate}

 Now, by the regularity estimates of Proposition \ref{Pr:Regularity} and by standard Schauder estimates, it is natural to expect that 
 \begin{equation}
 \Vert I_2\Vert\leq M \left\Vert \Phi \cdot \nu \right\Vert_{L^2(\partial E)}^2.
 \end{equation}
 To prove that the same estimate holds for $I_1$, it suffices, by continuity of the trace, to obtain 
 \begin{equation}\iint_\OT |\n p'|^2\leq M \int_{\partial E^*}(\Phi\cdot\nu)^2.\end{equation}  However, this just follows from standard parabolic estimates, provided we can prove that 
 \begin{equation}\label{Eq:Guan}
 \iint_\OT (u')^2\leq M \int_{\partial E^*}(\Phi\cdot\nu)^2.\end{equation} To prove \eqref{Eq:Guan}, we use $u'$ as a test function in the weak equation on $u'$    and obtain, by the Cauchy-Schwarz inequality and the continuity of the trace,
 \begin{align}
 \frac{\partial}{\partial t}\int_\O (u')^2(t,\cdot)+\int_\O |\n u'|^2&=\int_{\partial E}(\Phi\cdot \nu)u'
 \\&\leq M\Vert \Phi \cdot\nu\Vert_{L^2(\partial E)}\Vert \n u'\Vert_{L^2(\O)}.
 \end{align}
 Integrating this inequality in time yields the required result and we hence obtain 
 \begin{equation}\left|\mathcal L_{\B^*}''(E)[\Phi,\Phi]\right|\leq M\Vert \Phi\cdot \nu \Vert_{L^2(\partial E)}^2.\end{equation}
 As a consequence, the $L^2$ norm should be the optimal coercivity norm.

\subsection{Quantitative bathtub principle: using the convexity of the functional}

In this section, we will fully exploit the convexity of the functional. We first heuristically explain how we are going to make use of it.
\paragraph{Heuristics}
Let us assume that we are working with a competitor $f\in {\overline {\mathcal M}}(\O)$, and let us define $p_f$ as the adjoint state associated to $f$ (solution of \eqref{Eq:Adjoint}). Hence, for an admissible perturbation $h$ at $f$ (i.e, such that $f+th\in \overline{\mathcal M}(\O)$ for any $t\geq 0$ small enough), the derivative of $\mathcal J_T$ at $f$ in the direction $h$ is given by (Proposition \ref{Pr:Adjoint})
\begin{equation}\dot{\mathcal J_T}(f)[h]=\int_\O h(x)\left(\int_0^T p_f(t,x)dt\right)dx.\end{equation}
Since $\mathcal J_T$ is convex (Proposition \ref{Pr:Convexity}), we have
\begin{equation}\label{Eq:Heuristics}
\mathcal J_T(f+h)-\mathcal J_T(f)\geq \dot{\mathcal J_T}(f)[h].
\end{equation}
As a consequence, let us assume that $f=\mathds 1_E$. In order to maximise the right hand side of \eqref{Eq:Heuristics}, we need to choose $h$ such that, defining $\overline \Psi:=\int_0^T p_f(t,\cdot)dt$, and choosing $\overline c>0$ such that $\operatorname{Vol}\left(\{\overline \Psi>\overline c\}\right)=V_0$ (assuming this set is uniquely defined and regular), 
\begin{equation}f+h=\mathds 1_{\{\overline \Psi>\overline c\}}=\mathds 1_{\overline E}\end{equation} and so we obtain the lower bound 
\begin{equation}\label{Eq:FWF}\mathcal J_T(\mathds 1_{\overline E})-\mathcal J_T(\mathds 1_E)\geq \int_{\overline E} \overline \Psi-\int_{E}\overline \Psi.\end{equation} Now, as we will see, when $f$ is close enough to $f^*$, $\overline E$ should be a normal deformation of $\B^*$, and the only thing left is thus to quantify 
\begin{equation}\label{Eq:Start}
 \int_{\overline E} \overline \Psi-\int_{E}\overline \Psi.
\end{equation}
Indeed, using \eqref{Eq:FWF} we obtain
\begin{equation}\mathcal J_T(\mathds 1_{\B^*})-\mathcal J_T(\mathds 1_E)\geq J_T(\mathds 1_{\B^*})-\mathcal J_T(\mathds 1_{\overline E})+\mathcal J_T(\mathds 1_{\overline E})-\mathcal J_T(\mathds 1_{ E})\geq C \operatorname{Vol}(E^*\Delta \overline E)^2+\mathcal J_T(\mathds 1_{\overline E})-\mathcal J_T(\mathds 1_{ E}).\end{equation} Here, $C>0$ is given by Proposition \ref{Pr:NormalDeformation}.

Since $\overline f:=\mathds 1_{\overline E}$ is a maximiser of $T_{\overline \Psi}:f\mapsto \int_\O f\overline \Psi$ in $\overline{\mathcal M}(\O)$, it turns out that estimating \eqref{Eq:Start} amounts to providing a quantitative estimate for the linear optimisation problem 
\begin{equation}\sup_{f\in \overline{\mathcal M}(\O)}T_{\overline \Psi}(f)\end{equation} which is exactly the quantitative version of the bathtub principle.

The goal of the present paragraph is to give a uniform bathtub principle that was presented in a slightly different form in the section devoted to Theorem \ref{Th:Td}, see Proposition \ref{Pr:BathtubTD} above.

\begin{proposition}\label{Pr:Bathtub}
Let $\beta>0$ and let $\{\psi_i\}_{i\in I}\subset  \mathscr C^{1,\beta}(\O;\R_+)^I$ be a closed subset of $\mathscr C^{1,\beta'}(\O)$ for some $\beta'<\beta$. We assume that:
\begin{enumerate}
\item For every $i\in I$ there exists a unique $c_i$ such that, up to a set of measure 0, 
\begin{equation}\O_i:=\{\psi_i>c_i\}=\{\psi_i\geq c_i\}\end{equation} and 
\begin{equation}\label{Eq:Downtown}\forall i \in I\,, \operatorname{Vol}(\O_i)=V_0\,, \overline L=\sup_{i\in I}\operatorname{Per}(\O_i)<+\infty.\end{equation}
We define, for any $i\in I$, 
$$\overline f_i:=\mathds 1_{\O_i}.$$
\item There exists $M>0$ such that 
\begin{equation}\sup_{i\in I}\Vert \psi_i\Vert_{\mathscr C^{1,\beta}}\leq M_I.\end{equation}
\item There exists $\underline \mu>0$ such that 
\begin{equation}\label{Eq:Feather}\inf_{i\in I}\inf_{\partial \O_i}\left\{-\frac{\partial \psi_i}{\partial\nu}\right\}\geq \underline \mu.\end{equation}
\end{enumerate}
Then there exists a constant $\overline \omega>0$  such that \begin{equation}
\forall i\in I\,, \forall f\in {\overline{\mathcal M}}(\O)\,, \int_\O (\overline f_i-f)\psi_i\geq \overline \omega \Vert \overline f_i-f\Vert_{L^1(\O)}^2.
\end{equation}
\end{proposition}
The proof of this Proposition is very similar to that of Proposition \ref{Pr:BathtubTD}.
\begin{proof}[Proof of Proposition \ref{Pr:Bathtub}] 
We just need to prove that, thanks to our assumption, we can bring ourselves back to the proof of Proposition \ref{Pr:BathtubTD}. This is done using the Schwarz rearrangement, as was done in \cite{MRB2020}.

From the bathtub principle we have, for any $i\in I$, that $f_i^*:=\mathds 1_{\O_i}$ is the unique solution of 
\begin{equation}\sup_{f\in \mathcal M(\O)}\int_\O f\psi_i.\end{equation}

By the uniform H\"{o}lder continuity of $\{\n \psi\}_{i\in I}$ there exists $\overline \e>0$ that only depends on $M_I$ and $\underline \mu$ such that
\begin{equation}\label{Eq:OS}\forall x \in \O\,, \forall i\in I\,, \psi_i(x)\in \Big(c_i-\overline \e;c_i+\overline \e\Big)\Rightarrow \left|\n \psi_i\right|(x)\geq \frac{\underline \mu}2\,, \sup_{\e\in (-\overline \e;\overline \e)}\sup_{i\in I}\operatorname{Per}(\{\psi_i=c_i+\e\})<+\infty.\end{equation} Let us fix such an $\overline \e$.

We now reduce ourselves to the case of radially symmetric function:
\paragraph{Reduction to radially symmetric functions}
For any $i\in I$, let us consider the distribution function $\mathscr L_i$ of $\psi_i$. From \eqref{Eq:OS}, $\mathscr L_i$ is $\mathscr C^1$ in  $(c_i-\overline \e;c_i+\overline \e)$ and so, letting $\psi_i^\#$ be the Schwarz rearrangement of $\psi_i$, we have 
$$\int_{\partial \{\psi_i>c_i\}} \frac1{\left|\frac{\partial \psi_i}{\partial \nu}\right|}=-\mathscr L_i'(c_i)=\int_{\partial \B(0;r^*)}\frac1{\left|\frac{\partial \psi_i^\#}{\partial \nu}\right|}.$$
Given the uniform perimeter bound \eqref{Eq:OS} on the level sets close to $\{\psi_i=c_i\}$, it thus follows that there exists a constant $C>0$  and $\underline \e>0$ such that $\{\psi_i^\dagger\}_{i\in I}$ satisfies, in a $(r^*-\underline \e;r^*+\underline \e)$,

\begin{equation}\label{Eq:V}\forall i\in I\,, \left|\frac{d\psi_i^\dagger}{dr}\right|\geq C\underline \mu. \end{equation} 
We can then observe the following thing: by equimeasurability of the Schwarz rearrangement, we have, for every $f\in \mathcal M(\O)$, the following property: if $\Vert f-\overline f_i\Vert_{L^1(\O)}=\delta$ then, defining $\mathbb A_\delta$ as the unique annulus such that $\operatorname{Vol}(\mathbb A_\delta)=V_0\,, \operatorname{Vol}(\mathbb A_\delta \Delta \B^*)=\delta$, the Haryd-Littlewood inequality and the equimeasurability of the rearrangement ensure that
\begin{equation}
\int_\O (\overline f_i-f)\psi_i\geq \int_\O (f^*-\mathds 1_{\mathbb A_\delta})\psi_i^\#.\end{equation} Hence, it suffices to prove that 
\begin{equation}\int_\O (f^*-\mathds 1_{\mathbb A_\delta})\psi_i^\#\geq \underline \omega \delta^2\end{equation} where $\underline \omega$ does not depend on $i$.   Thanks to \eqref{Eq:V}, the rest of the proof follows along the same exact lines as Proposition \ref{Pr:BathtubTD}.
\end{proof}

\subsection{Combining the bathtub principle and shape derivatives}
To conclude the proof of Theorem \ref{Th:Ti}, it thus only remains to prove the following proposition:

\begin{proposition}\label{Pr:Le}
Estimate \eqref{Eq:Le} holds.\end{proposition}
\begin{proof}\label{Pr:Le}
We argue by contradiction and assume that Estimate \eqref{Eq:Le} does not hold. Let us then consider a sequence $\{f_k\}_{k\in \N}$ such that 
\begin{equation}\frac{\mathcal J_T(f^*)-\mathcal J_T(f_k)}{\Vert f_k-f^*\Vert_{L^1(\O)}^2}\underset{k\to \infty}\rightarrow 0.\end{equation}
As in the proof of Theorem \ref{Th:Td}, the only closure point of $\{f_k\}_{k\in \N}$ (in a weak $L^\infty-*$ sense) is $f^*$. We introduce, for any $k\in \N$, 
\begin{equation}\delta_k:=\Vert f_k-f^*\Vert_{L^1(\O)}.\end{equation} Up to replacing $f_k$ with $f_{\delta_k}$, we can assume that $f_k=f_{\delta_k}$.

Let us define, for any $k\in \N$, $p_k$ as the adjoint state (solution of \eqref{Eq:Adjoint}) with $f=f_k$. From standard parabolic regularity and Proposition \ref{Pr:Regularity}, for any $\beta\in(0;1)$ there exists $M_\beta$ such that for any $t\in [0;T]$
\begin{equation}
\Vert p_k\Vert_{\mathscr C^{2,\beta}(\O)}\leq M_\beta,
\end{equation}
and hence, since $f_k\underset{k\to \infty}\rightarrow f^*$, we obtain
\begin{equation}
p_k\underset{k\to \infty}{\overset{\mathscr C^{2,\beta}(\OT)}\rightarrow} p^*\end{equation}where $p^*$ is the adjoint state associated with $f=f^*$.

Let, for any $k\in \N$, 
\begin{equation}\Psi_k:=\int_0^T p_k.\end{equation} From the same arguments, 
\begin{equation}\Psi_k\underset{k\to \infty}{\overset{\mathscr C^{2,\beta}(\O)}\rightarrow}\Psi^*:=\int_0^T p^*.\end{equation} $\Psi^*$ is radially symmetric, it is decreasing and its only level set of volume $V_0$ is $\mathbb B^*$.  Furthermore, from the same arguments as in Claim \ref{Cl:NDA},  we also have 
\begin{equation}\forall \eta>0\,, \inf_{r>\eta}\left|\frac{\partial \Psi^*}{\partial r}\right|(r^*)=\ell(\eta)>0.\end{equation} Let, for any $k\in \N$, $c_k$ be such that $\operatorname{Vol}\left(\{\Psi_k\geq c_k\}\right)\geq V_0\,, \operatorname{Vol}\left(\{\Psi_k>c_k\}\right)\leq V_0$.

Since $f_k\underset{k\to \infty}\rightarrow f^*$ and since $\mathcal J_T(f_k)\geq\mathcal J_T(\mathds 1_{\{\Psi_k>c_k\}})$ by convexity of the functional, it follows that $\left\{\mathds 1_{\{\Psi_k>c_k\}}\right\}_{k\in \N}$ converges weakly to $f^*$. Since $f^*$ is an extreme point of $\overline{\mathcal M}(\O)$, this convergence occurs in $L^1$. 
We choose $\eta>0$ small enough so that, for any $k\in \N$ large enough, $\{\psi_k=c_k\}\cap \{0<r<\eta\}=\emptyset$. This is possible because $\Psi^*$ is radially symmetric and decreasing: indeed, argue by contradiction and assume that there exists a sequence $\{x_k\}_{k\in \N}$ converging to 0 such that for any $k\in \N$, $\Psi_k(x_k)=c_k$. Let $c$ be the limit of the sequence $\{c_k\}$. Since $\operatorname{Vol}(\{\Psi_k>c_k\})=V_0$, there exists $\eta'>0$ such that for any $k\in \N$ there exists $y_k\,, \Vert y_k\Vert>\eta'$ such that $\Psi_k(y_k)>c_k$. Passing to the limit, there exists $y'>0$ such that $\Psi^*(y')\geq c=\lim_{k\to \infty}\Psi_k(x_k)=\Psi^*(0)$ and so $\Psi^*$ can not be decreasing. Hence such an $\eta>0$ exists.

 As a consequence, for such an $\eta$ we have, for any $k$ large enough, 
 \begin{equation}\inf_{x\,, \Vert x\Vert>\eta}\left| \n \Psi_k\right|(x)\geq \frac{\ell(\eta)}2.\end{equation} Thus, the level set $\{\Psi_k=c_k\}$ is a $\mathscr C^1$ curve and 
 \begin{equation}\inf_{\{\Psi_k=c_k\}}|\n \Psi_k|\geq\frac{\ell(\eta)}2.\end{equation}    Since $\{\Psi_k\}_{k\in \N}$ is uniformly bounded in $\mathscr C^{2}(\O)$, these sets have uniformly Lipschitz  boundaries.
 It follows that the sequence of sets $\left\{\{\Psi_k>c_k\}\right\}_{k\in \N}$ converges in Hausdorff distance to $\{\Psi^*>c^*\}=\B^*$ where $c^*=\Psi^*(r^*)$.

 Finally, for any $k\in \N$ large enough, $E_k:=\{\Psi_k=c_k\}$ is a normal deformation of $\B^*$. Indeed, assuming that it is not, there exists a sequence $\{x_k\}_{k\in \N}\in( \partial \B^*)^\N$ and two sequences $\{t_{i,k}\}_{i=1,2,k\in \N}$ converging to 0 such that $\Psi_k(x_k+t_{i,k}\nu(x_k))=c_k$, $i=1,2$. This gives the existence of a $t_k$, converging to 0 as $k\to \infty$, such that $\langle \n \Psi_k(x_k+t_k\nu(x_k)),\nu(x_k)\rangle=0$, which yields a contradiction when passing to the limit. Thus, $\partial E_k$ converges $W^{2,p}$ to $\partial \B^*$ for all $p>1$, and in $\mathscr C^{2}$, $\beta\in (0;1)$, and the sequence $\{\operatorname{Per}(\{\Psi_k=c_k\})\}_{k\in \N}$ is bounded.

 We can now prove Estimate \eqref{Eq:Le}: from the convexity of the functional and the fact that, for any $k\in \N$,  $f_k$ solves $(P_{\delta_k})$, there exists a subset $F_k$ of $\O$ such that $f_k=\mathds 1_{F_k}$. Let $E_k=\{\Psi_k>c_k\}$ be the unique level-set of $\Psi_k$ of measure $V_0$. By convexity of the functional, 
 \begin{equation}\mathcal J_T(\mathds 1_{E_k})-\mathcal J_T(f_k)\geq \int_\O (\mathds 1_{E_k}-\mathds 1_{F_k})\Psi_k.\end{equation}
 
 We are now in a position to apply Proposition \ref{Pr:Bathtub}: with the $\overline \omega$ given in Proposition \ref{Pr:Bathtub}, we thus have 
 
\begin{equation}\mathcal J_T(\mathds 1_{E_k})-\mathcal J_T(f_k)\geq \overline \omega \operatorname{Vol}(F_k\Delta E_k)^2.\end{equation}

Then, as $E_k$ is a normal deformation of $\B^*$ we can apply Proposition \ref{Pr:NormalDeformation} and obtain, for $C>0$ given by Proposition \ref{Pr:NormalDeformation}, 
\begin{equation}\mathcal J_T(\mathds 1_{\B^*})-\mathcal J_T(\mathds 1_{E_k})\geq C\operatorname{Vol}(E_k\Delta \B^*)^2.\end{equation}
We obtain the existence a $C'>0$ such that 
\begin{equation}\mathcal J_T(\mathds 1_{\B^*})-\mathcal J_T(\mathds 1_{E_k})\geq C'\left( \operatorname{Vol}(E_k\Delta \B^*)^2+\operatorname{Vol}(F_k\Delta E_k)^2\right).\end{equation} However, by the triangle inequality in $L^1$ and the arithmetic-geometric inequality,
\begin{equation}\operatorname{Vol}(\B^*\Delta F_k)^2\leq \frac12\left(\operatorname{Vol}(F_k\Delta E_k)^2+\operatorname{Vol}(E_k\Delta \B^*)^2\right).\end{equation}The conclusion follows.

\end{proof}

\section{Conclusion}\label{Se:Concl}
\subsection{Structure of the problem, structure of the proof}
In this paper, we have investigated two possible approaches to quantitative inequalities for time-evolving optimal control problems. While Theorem \ref{Th:Td}, dealing with time-dependent controls, is more powerful than Theorem \ref{Th:Ti}, it is likely that its proof does not generalise easily to other domains. Indeed, the first step of the proof is to identify, explicitly, the maximisers of an auxiliary optimisation problem, which can not be done in general, non-spherical domains.

On the other hand, the proof of Theorem \ref{Th:Ti} is susceptible of applying to other cases. Let us specify what we mean: considering a controlled heat equation 
\begin{equation}\frac{\partial u_f}{\partial t}-\Delta u_f=f\end{equation} with Dirichlet boundary conditions, and where $f\in \overline{\mathcal M}(\O)$, let $f^*$ be a solution of \eqref{Eq:PvTi}. The convexity of the functional $\mathcal J_T$ (Proposition \ref{Pr:Convexity}) holds independently of the geometry of the domain an so any maximiser $f^*$ writes $\mathds 1_{E^*}$ for some subset $E^*$ of $\O$. In order to carry out the proof of Theorem \ref{Th:Ti} in this new domain, several things are in order:
\begin{enumerate}
\item The regularity of optimal sets: each set $E^*$ such that $f^*=\mathds 1_{E^*}$ is a solution of \eqref{Eq:PvTi} needs to be smooth enough that shape derivatives of the criterion may be computed. It is unclear at this stage whether or not the classical regularity works valid in the stationary case may be applied to obtain such regularity.
\item The coercivity of shape Lagrangians: defining $I^*:=\{E^*\subset \O\,, \mathds 1_{E^*}\text{ solves \eqref{Eq:PvTi}}\}$ and assuming that each $E^*\in I^*$ is smooth enough to compute first and second order shape derivatives, one needs to check that, defining the Lagrangian $L_{E^*}$ associated with the volume constraint, there exists a constant $\alpha>0$ such that, for any $E^*\in I^*$ and any admissible vector field $\Phi$ at $E^*$, there holds
\begin{equation}
L_{E^*}(E^*)[\Phi,\Phi]\geq \alpha \Vert \Phi \cdot \nu\Vert_{L^2(\partial E^*)}^2.\end{equation} This kind of estimate seems to be extremely challenging to obtain in general, as indicated by the fact that, in this paper, such a coercivity was obtained by explicit diagonalisation of the shape hessian. Such diagonalisation may not be available in general.
\end{enumerate}
If these two assumptions are satisfied, then we believe that the method of proof of Theorem \ref{Th:Ti} may adapt.

\subsection{The optimal coercivity norm for other types of constraints}\label{Cl:Time}
As mentioned in the Introduction, an interesting question is that of knowing whether or not the coercivity norm obtained in Theorem \ref{Th:Td} remains unchanged when considering other types of $L^1$ constraints. Indeed, let us consider the following variation: defining 
\begin{equation}\tilde{\mathcal M}(\O):=\left\{f\in L^\infty(\OT)\,, 0\leq f\leq 1\text{ a.e.,} \iint_\OT f=V_0\right\}\end{equation} we investigate the optimisation problem
\begin{equation}\sup_{f\in \tilde{\mathcal M}(\O)}J_T(f).\end{equation} Here, the convexity of the functional $\mathcal J_T$ is still valid, so that a solution of this new problem writes $f^*=\mathds 1_{E^*}$, with $E^*$ a measurable subset of $E^*$. Then, if one were to compare $f^*$ with a competitor $f=\mathds 1_E$, it would be more natural to expect the \textquotedblleft classical \textquotedblright discrepancy norm
\begin{equation}\operatorname{Vol}(E^*\Delta E)^2=\left(\iint_\OT \left|f-f^*\right|\right)^2\end{equation} to be optimal. We do not believe this to be true, however, and we believe that the correct discrepancy norm remains 
\begin{equation}\int_0^T\operatorname{Vol}(E^*(t)\Delta E(t))^2dt.\end{equation}
To give some explanation as to why we believe this is to be expected, we can once again consider the case of the ball $\O=\mathbb B(0;R)$. Once again, the rearrangement arguments used throughout the paper remain valid, and, for almost every $t\in (0;T)$, $E^*(t)$ is a centred ball of radius $r(t)\geq 0$. We expect several difficulties in treating this problem (most notably, we expect the (non)-degeneracy of $r$, or, in other terms, the control of the set $\{r=0\}$, to be very hard to obtain) but the methods of Theorem \ref{Th:Td} should once again provide a quadratic estimate at each time $t$, yielding the aforementioned stronger estimate. We underline once again that, at the present moment, it is unclear to us how one may fully analyse this type of global constraint.

\subsection{Theorem \ref{Th:Td}: on the Assumption $\e>0$}
One may also argue that the assumption $\e>0$ is artificial. At this stage, and since we use in a crucial manner the uniform non-degeneracy of the switch function (Claim \ref{Cl:NDA}), we are not yet in a position to give a proof that would bypass this assumption. However, it should be noted that our proof makes use of very strong regularity properties in order to derive the uniform bathtub principle. It would be interesting to see if, using the general quantitative Hardy-Littlewood inequality {\cite{Cianchi2008}} one could bypass the strength required in the present proof to obtain the case $\e=0$ (and, in general, it would be extremely interesting to use {\cite{Cianchi2008}} to see if Theorem \ref{Th:Td} could be obtained in more general domain).
\subsection{Using the quantitative isoperimetric inequality to obtain our results}\label{SD}

We touch on another way which it would be interesting to investigate, that of using the quantitative isoperimetric inequality in order to obtain Theorem \ref{Th:Td}. It would amount, in the approach of \cite{RakotosonMossino} (see Section \ref{Se:Uniq}), to supplementing the isoperimetric inequality in \eqref{Eq:Co}. We expect that this would lead to a control of the isoperimetric deficit $\mathcal A(t,\tau)$ of the level sets $\{u_f(t,\cdot)>\tau\}$ in the sense that we could give a lower bound of the form $\int_0^T \int_0^{\Vert u_f\Vert_{L^\infty}}\mathcal A(t,\tau)^2d\tau dt$, but is unclear how this would then translate to a control of the isoperimetric deficit of $f$.

\subsection{Minimisation problems}
We believe the proof for minimisation problems works in exactly the same way, as we also have an explicit description of minimisers using rearrangement techniques.
\subsection{Technical obstructions and possible generalisations for bilinear control problems}\label{Cl:Bilinear}

Finally, we touch upon bilinear control problems. Let us assume that we are working with the state equation 
\begin{equation}\begin{cases}
\frac{\partial u_f}{\partial t}-\Delta u_f=fu_f\text{ in }\OT\,, 
\\ u_f(t=0)=u_0\geq 0\,, u_0\neq 0\,, 
\\ u_f(t,\cdot)=0\text{ on }(0;T)\times \partial \O.\end{cases}\end{equation} The maximisation problem reads the same:
\begin{equation}\sup_{f\in \mathcal M_T(\O)}\frac12\iint_\OT u_f^2.\end{equation} Here we con once again explicitly characterise the maximisers using rearrangement techniques. However: the convexity of the functional is no longer obvious, and it can be checked that the switch function is here given 
\begin{equation}\Psi=p_fu_f\end{equation} where $p_f$ solves
\begin{equation}\begin{cases}
\frac{\partial p_f}{\partial t}+\Delta p_f=-fp_f-u_f\,, 
\\ p_f(T,\cdot)=0\,, 
\\ p_f(t,\cdot)=0\text{ on }(0;T)\times \partial \O.\end{cases}\end{equation} Here we see our first difference with our approach, which is that the switch function can merely be expected to be $\mathscr C^{0,\alpha}$, which is in contrast with the $\mathscr C^{2,\alpha}$ regularity we obtained in our paper. Maybe it is possible to bypass this problem using the tools of \cite{Cianchi2008}.

\bibliographystyle{abbrv}
\bibliography{BiblioTimeEvolving-2020}

\begin{thebibliography}{10}

\bibitem{Alabau2}
F.~Alabau-Boussouira.
\newblock Insensitizing exact controls for the scalar wave equation and exact
  controllability of {\textdollar}{\textdollar}2{\textdollar}{\textdollar}
  -coupled cascade systems of {PDE}'s by a single control.
\newblock {\em Mathematics of Control, Signals, and Systems}, 26(1):1--46, May
  2013.

\bibitem{Alvino1986}
A.~Alvino, P.~Lions, and G.~Trombetti.
\newblock A remark on comparison results via symmetrization.
\newblock {\em Proceedings of the Royal Society of Edinburgh: Section A
  Mathematics}, 102(1-2):37--48, 1986.

\bibitem{alvino1991}
A.~Alvino, P.-L. Lions, and G.~Trombetti.
\newblock Comparison results for elliptic and parabolic equations via
  symmetrization: a new approach.
\newblock {\em Differential Integral Equations}, 4(1):25--50, 1991.

\bibitem{Alvino1989}
A.~Alvino, G.~Trombetti, and P.~Lions.
\newblock On optimization problems with prescribed rearrangements.
\newblock {\em Nonlinear Analysis: Theory, Methods {\&} Applications},
  13(2):185--220, Feb. 1989.

\bibitem{Alvino1990}
A.~Alvino, G.~Trombetti, and P.-L. Lions.
\newblock Comparison results for elliptic and parabolic equations via schwarz
  symmetrization.
\newblock {\em Annales de l'Institut Henri Poincare (C) Non Linear Analysis},
  7(2):37--65, Mar. 1990.

\bibitem{Bandle}
C.~Bandle.
\newblock {\em Isoperimetric Inequalities and Applications}.
\newblock Monographs and studies in mathematics. Pitman, 1980.

\bibitem{Barchiesi2014}
M.~Barchiesi, G.~M. Capriani, N.~Fusco, and G.~Pisante.
\newblock Stability of p{\'{o}}lya{\textendash}szeg{\H{o}} inequality for
  log-concave functions.
\newblock {\em Journal of Functional Analysis}, 267(7):2264--2297, Oct. 2014.

\bibitem{BrascoButtazzo}
L.~Brasco and G.~Buttazzo.
\newblock Improved energy bounds for schr\"{o}dinger operators.
\newblock {\em Calculus of Variations and Partial Differential Equations},
  53(3-4):977--1014, Sept. 2014.

\bibitem{BDPV}
L.~Brasco, G.~De~Philippis, and B.~Velichkov.
\newblock Faber--krahn inequalities in sharp quantitative form.
\newblock {\em Duke Math. J.}, 164(9):1777--1831, 06 2015.

\bibitem{Brothers1988}
Z.~W.~P. Brothers, John~E.
\newblock Minimal rearrangements of sobolev functions.
\newblock {\em Journal f{\"u}r die reine und angewandte Mathematik},
  384:153--179, 1988.

\bibitem{Harbrecht}
R.~Br\"{u}gger, H.~Harbrecht, and J.~Tausch.
\newblock On the numerical solution of a time-dependent shape optimization
  problem for the heat equation, 06 2019.

\bibitem{CarlenLieb}
E.~A. Carlen, R.~L. Frank, and E.~H. Lieb.
\newblock Stability estimates for the lowest eigenvalue of a schr\"{o}dinger
  operator.
\newblock {\em Geometric and Functional Analysis}, 24(1):63--84, Feb. 2014.

\bibitem{Cass1966}
D.~Cass.
\newblock Optimum growth in an aggregative model of capital accumulation: A
  turnpike theorem.
\newblock {\em Econometrica}, 34(4):833, 1966.

\bibitem{CianchiEsposito}
A.~Cianchi, L.~Esposito, N.~Fusco, and C.~Trombetti.
\newblock A quantitative p{\'o}lya-szeg{\"o} principle.
\newblock 2008.

\bibitem{Cianchi2008}
A.~Cianchi and A.~Ferone.
\newblock A strengthened version of the hardy-littlewood inequality.
\newblock {\em Journal of the London Mathematical Society}, 77(3):581--592,
  Feb. 2008.

\bibitem{DambrineLamboley}
M.~Dambrine and J.~Lamboley.
\newblock Stability in shape optimization with second variation.
\newblock {\em Journal of Differential Equations}, 267(5):3009--3045, Aug.
  2019.

\bibitem{Dorfman}
R.~Dorfman, P.~Samuelson, and E.~Solow.
\newblock {\em Linear programming and economic analysis}.
\newblock New York, McGraw-Hill, 1958.

\bibitem{LissyPrivatErvedoza}
S.~Ervedoza, P.~Lissy, and Y.~Privat.
\newblock {Insensitizing controls for the heat equation with respect to
  boundary variations}.
\newblock working paper or preprint, Dec. 2020.

\bibitem{borjan}
C.~Esteve, B.~Geshkovski, D.~Pighin, and E.~Zuazua.
\newblock {Large-time asymptotics in deep learning}.
\newblock working paper or preprint, Aug. 2020.

\bibitem{Ferone2003}
V.~R. Ferone, Adele.
\newblock Minimal rearrangements of sobolev functions : a new proof.
\newblock {\em Annales de l'I.H.P. Analyse non lin{\'e}aire}, 20(2):333--339,
  2003.

\bibitem{FuscoMaggiPratelli}
N.~Fusco, F.~Maggi, and A.~Pratelli.
\newblock The sharp quantitative isoperimetric inequality.
\newblock {\em Annals of Mathematics}, 168(3):941--980, Nov. 2008.

\bibitem{HenrotPierre}
A.~Henrot and M.~Pierre.
\newblock {\em Shape Variation and Optimization}.
\newblock European Mathematical Society Publishing House, Feb. 2018.

\bibitem{Kawohl}
B.~Kawohl.
\newblock {\em Rearrangements and Convexity of Level Sets in {PDE}}.
\newblock Springer Berlin Heidelberg, 1985.

\bibitem{Kesavan1988}
S.~Kesavan.
\newblock Some remarks on a result of talenti.
\newblock {\em Annali della Scuola Normale Superiore di Pisa - Classe di
  Scienze}, 15(3):453--465, 1988.

\bibitem{Kesavan}
S.~Kesavan.
\newblock {\em Symmetrization and Applications}.
\newblock {WORLD} {SCIENTIFIC}, Apr. 2006.

\bibitem{LanceTrelatZuazua}
G.~Lance, E.~Tr{\'{e}}lat, and E.~Zuazua.
\newblock Shape turnpike for linear parabolic {PDE} models.
\newblock {\em Systems {\&} Control Letters}, 142:104733, Aug. 2020.

\bibitem{Lieberman}
G.~Lieberman.
\newblock {\em Second Order Parabolic Differential Equations}.
\newblock World Scientific, 1996.

\bibitem{LissyPrivat}
P.~Lissy, Y.~Privat, and Y.~Simpor{\'{e}}.
\newblock Insensitizing control for linear and semi-linear heat equations with
  partially unknown domain.
\newblock {\em {ESAIM}: Control, Optimisation and Calculus of Variations},
  25:50, 2019.

\bibitem{MazariQuantitative}
I.~Mazari.
\newblock Quantitative inequality for the eigenvalue of a schr\"{o}dinger
  operator in the ball.
\newblock {\em Journal of Differential Equations}, 269(11):10181--10238, Nov.
  2020.

\bibitem{MazariNadinPrivat}
I.~Mazari, G.~Nadin, and Y.~Privat.
\newblock Optimization of a two-phase, weighted eigenvalue with dirichlet
  boundary conditions.
\newblock Preprint, 2019.

\bibitem{MRB2020}
I.~Mazari and D.~Ruiz-Balet.
\newblock Quantitative stability for eigenvalues of schr\"{o}dinger operator,
  quantitative bathtub principle and application to the turnpike property for a
  bilinear optimal control problem.
\newblock Preprint, 2020.

\bibitem{Mossino}
J.~Mossino.
\newblock {\em In{\'e}galit{\'e}s isop{\'e}rim{\'e}triques et applications en
  physique}.
\newblock Hermann Paris, 1984.

\bibitem{RakotosonMossino}
J.~Mossino and J.~M. Rakotoson.
\newblock Isoperimetric inequalities in parabolic equations.
\newblock {\em Annali della Scuola Normale Superiore di Pisa - Classe di
  Scienze}, Ser. 4, 13(1):51--73, 1986.

\bibitem{Moubachir2006}
M.~Moubachir and J.-P. Zolesio.
\newblock {\em Moving Shape Analysis and Control}.
\newblock Chapman and Hall/{CRC}, Jan. 2006.

\bibitem{Souplet}
P.~Quittner and P.~Souplet.
\newblock {\em Superlinear parabolic problems : blow-up, global existence and
  steady states}.
\newblock Birkhauser, Basel Boston, 2007.

\bibitem{Rakotoson}
J.-M. Rakotoson.
\newblock {\em R{\'{e}}arrangement Relatif}.
\newblock Springer Berlin Heidelberg, 2008.

\bibitem{PighinSakamoto}
N.~Sakamoto, D.~Pighin, and E.~Zuazua.
\newblock The turnpike property in nonlinear optimal control {\textemdash} a
  geometric approach.
\newblock In {\em 2019 {IEEE} 58th Conference on Decision and Control ({CDC})}.
  {IEEE}, Dec. 2019.

\bibitem{Talenti}
G.~Talenti.
\newblock Elliptic equations and rearrangements.
\newblock {\em Annali della Scuola Normale Superiore di Pisa - Classe di
  Scienze}, Ser. 4, 3(4):697--718, 1976.

\bibitem{TrelatZhang}
E.~Tr{\'{e}}lat and C.~Zhang.
\newblock Integral and measure-turnpike properties for infinite-dimensional
  optimal control systems.
\newblock {\em Mathematics of Control, Signals, and Systems}, 30(1), Mar. 2018.

\bibitem{TrelatZhangZuazua}
E.~Tr{\'{e}}lat, C.~Zhang, and E.~Zuazua.
\newblock Steady-state and periodic exponential turnpike property for optimal
  control problems in hilbert spaces.
\newblock {\em {SIAM} Journal on Control and Optimization}, 56(2):1222--1252,
  Jan. 2018.

\bibitem{Vazquez}
J.~L. Vazquez.
\newblock Sym{\'e}trisation pour $u_t={\Delta}\varphi(u)$ et applications.
\newblock {\em C. R. Acad. Sci. Paris S{\'e}r. I Math.}, 295, 1982.

\bibitem{2006}
A.~Zaslavski.
\newblock {\em Turnpike Properties in the Calculus of Variations and Optimal
  Control}.
\newblock Springer-Verlag, 2006.

\bibitem{Zuazua2017}
E.~Zuazua.
\newblock Large time control and turnpike properties for wave equations.
\newblock {\em Annual Reviews in Control}, 44:199--210, 2017.

\end{thebibliography}

\appendix
\section{Proof of technical lemmas}\label{Ap:Technical}

\subsection{Proof of Lemma \ref{Le:ExistenceDelta}}

\begin{proof}[Proof of Lemma \ref{Le:ExistenceDelta}]
This Lemma relies on two elements: the first one is the weak continuity of $\mathcal J_T$, given in the following claim
\begin{claim}\label{Cl:Continuity}
Assume $\{f_k\}_{k\in \N}\in \overline{\mathcal M}(\O)^\N$ converges weakly $L^\infty-*$ to $f_\infty$. Then 
\begin{equation}
\mathcal J_T(f_k)\underset{k\to \infty}\rightarrow \mathcal J_T(f_\infty).
\end{equation}
\end{claim}
\begin{proof}[Proof of Claim \ref{Cl:Continuity}]
This proof relies on standard parabolic estimates.
\end{proof}
The second element is the following property of the class $\overline{\mathcal M}(\O,\delta)$:
\begin{claim}\label{Cl:MdeltaCompact}
The class $\overline{\mathcal M}(\O,\delta)$ is weakly $L^\infty-*$ compact.
\end{claim}
\begin{proof}[Proof of Claim \ref{Cl:MdeltaCompact}]
For any $f\in \overline{\mathcal M}(\O,\delta)$, let us define 
$$h:=f-f^*.$$ The condition that $\int_\O=\int_\O f^*$ rewrites 
\begin{equation}\label{Eq:Hzero}
\int_\O h=0.\end{equation}
Since $f^*=\mathds 1_{\B^*}$ is the characteristic function of a set, we also have 
\begin{equation}\label{Eq:H1}h\leq 0\text{ in }\mathbb B^*\,, h\geq 0\text{ in }(\B^*)^c.\end{equation} This, and Equation \eqref{Eq:Hzero}, allows to rewrite the condition $\Vert f-f^*\Vert_{L^1(\O)}$ as 
\begin{equation}\label{Eq:H2}
\int_{\B^*} |h|=-\int_{\B^*}h=\int_{(\B^*)^c} h=\int_{(\B^*)}|h|=\frac{\delta}2.\end{equation} Finally, $-1\leq h\leq 1$ as a consequence of its definition.

Let us then consider a sequence $\{f_k\}_{k\in \N}\in \overline{\mathcal M}(\O,\delta)^\N$ and define, for any $k\in \N$, 
\begin{equation}
h_k:=f_k-f^*.\end{equation}
Since $\overline{\mathcal M}(\O)$ is compact for the weak $L^\infty-*$ convergence, let us assume that there exists $f_\infty\in \overline{\mathcal M}(\O)$ such that
\begin{equation}f_k \underset{k\to \infty}\rightharpoonup f_\infty\end{equation} and define $h_\infty=f_\infty-f^*$. It is clear that 
\begin{equation} h_k \underset{k\to \infty}\rightharpoonup h_\infty.\end{equation} Since \eqref{Eq:H1}-\eqref{Eq:H2} are satisfied by $h_k$ for every $k\in \N$, it follows that they are satisfied by $h_\infty$.  As a consequence, 
\begin{equation}\int_{\B^*} |h_\infty|=-\int_{\B^*}h_\infty=\int_{(\B^*)^c} h_\infty=\int_{(\B^*)}|h_\infty\left. \right| =\frac{\delta}2\end{equation} and so 
\begin{equation}f_\infty \in \overline{\mathcal M}(\O,\delta),\end{equation} so that the Claim follows.

\end{proof}
Thus, to conclude the proof of the Lemma, it suffices to consider a minimising sequence $\{f_k\}_{k\in \N}\in \overline{\mathcal M}(\O,\delta)^\N$ for the variational problem \eqref{Eq:PvDelta}. One can extract a wek $L^\infty-*$ converging subsequence that converges to $f_\infty\in \overline{\mathcal M}(\O,\delta)$, and the Claim \ref{Cl:Continuity} enables one to pass to the limit. As a conclusion we obtain 
\begin{equation}\mathcal J_T(f_\infty)=\min_{f\in \overline{\mathcal M}(\O,\delta)}\mathcal J_T(f).\end{equation}

\end{proof}

\subsection{Proof of Lemma \ref{Le:Intermediaire}}
\begin{proof}[Proof of Lemma \ref{Le:Intermediaire}]
 First of all, the fact that Theorem \ref{Th:Ti} implies the conclusion of Lemma \ref{Le:Intermediaire} is trivial. Conversely, assume Lemma \ref{Le:Intermediaire} holds. Let us define the functional 
\begin{equation}\mathcal G_T:\overline{\mathcal M}(\O)\backslash\{f^*\}\ni f\mapsto \frac{\mathcal J_T(f^*)-\mathcal J_T(f)}{\Vert f-f^*\Vert_{L^1(\O)}^2}.\end{equation} Proving Theorem \ref{Th:Ti} is equivalent to proving 
\begin{equation}\inf_{f\in \overline{\mathcal M}(\O)}\mathcal G_T(f)>0.\end{equation} We consider a minimising sequence $\{f_k\}_{k\in \N^*}$ for $\mathcal G_T$. Let us consider a closure point $f_\infty$ of this sequence. If $f_\infty\neq f^*$ then from Claim \ref{Cl:Continuity} we have $$\lim_{k\to \infty}\mathcal J_T(f^*)-\mathcal J_T(f_k)=\mathcal J_T(f^*)-\mathcal J_T(f^\infty)=A>0$$ where the last inequality is strict because $f^*$ is the unique maximiser of $\mathcal J_T$ (Theorem \ref{Th:Uniqueness})  and so, using the trivial bound $\Vert f-f^*\Vert_{L^1(\O)}\leq 2\operatorname{Vol}(\O)$ we obtain
\begin{equation}
\lim_{k\to \infty}\mathcal G_T(f_k)\geq  \frac{\mathcal J_T(f^*)-\mathcal J_T(f_\infty)}{2\operatorname{Vol}(\O)}>0.
\end{equation}
If on the other hand we have $f_\infty=f^*$ then, $f^*$ being an extreme point of the convex set $\overline{\mathcal M}(\O)$, the convergence $f_k\underset{k\to\infty}\rightharpoonup f^*$ is strong in $L^1(\O)$ (\cite[Proposition 2.2.1]{HenrotPierre}). As a consequence we can define the sequence 
\begin{equation}
\forall k \in \N\,, \delta_k:=\Vert f_k-f^*\Vert_{L^1(\O)}.
\end{equation}
It then follows that there holds
\begin{align}
\underset{k\to \infty}{\lim\inf }\, \mathcal G_T(f_k)&\geq \underset{k\to \infty}{\lim\inf}\frac{\mathcal J_T(f^*)-\mathcal J_T(f_{\delta_k})}{\delta_k^2}
\\&>0
\end{align} if Estimate \eqref{Eq:Le} holds, and this concludes the proof of the equivalence between the two results.
\end{proof}


\subsection{Proof of the coercivity estimate-Lemma \ref{Le:Control}}\label{Ap:Control}

 \begin{proof}[Proof of Lemma \ref{Le:Control}] To alleviate the proof, we first note that such a continuity is standard to prove for the term $C(\operatorname{Vol}(\B_{t\Phi}^*)-V_0)^2$ and we hence omit it.
Let us then define the function $\overline j_\Phi:=\mathcal L_{\B^*}$ and prove this estimate for this term. First of all, standard computations show that $\overline j_\Phi$ is twice differentiable in the sense of shapes. Furthermore, the second order shape derivatives at a given shape $E$ such that $E\cap \partial \O=\emptyset$, in the direction $\Phi$, where $\Phi\in W^{2,p}(\O;\R^2)$ is compactly supported in $\O$, is given by 
\begin{equation}\mathcal L_{\B^*}''(E)[\Phi,\Phi]=\iint_{(0;T)\times \partial E}p'\left(\Phi\cdot \nu\right)+\int_{\partial E}\left(\Phi\cdot \nu\right)^2\left(\mathscr H \int_{0}^T p_E+\int_0^T\frac{\partial p_E}{\partial \nu}\right)-\overline \Psi_{\partial \B^*}\int_{\partial E}\mathscr H \left(\Phi\cdot \nu\right)^2\end{equation} where:
\begin{enumerate}
\item $\mathscr H$ is the mean curvature of $\partial E$,
\item $p_E$ solves
\begin{equation}\begin{cases}
\frac{\partial p_E}{\partial t}+\Delta p_E=-u_E\text{ in }\OT\,, 
\\ p_E(T,\cdot)=0\,, 
\\ p_E(t,\cdot)=0\text{ on }(0;T)\times \partial \O,
\end{cases}
\end{equation}

\item $u'$ solves 
\begin{equation}\begin{cases}
\frac{\partial u'}{\partial t}-\Delta u'=0\text{ in }\OT\,, 
\\ u'(0,\cdot)=0\,, 
\\\llbracket \partial_\nu u'\rrbracket=-1\text{ on }(0;T)\times \partial E,
\\ u'(t,\cdot)=0\text{ on }(0;T)\times \partial \O,
\end{cases}
\end{equation}
\item and $p'$ solves

\begin{equation}\begin{cases}
\frac{\partial p'}{\partial t}+\Delta p'=-u'\text{ in }\OT\,, 
\\ p'(T,\cdot)=0\,, 
\\ p'(t,\cdot)=0\text{ on }(0;T)\times \partial \O,
\end{cases}
\end{equation}
\item and 
$$\left.\overline \Psi_{\partial \B^*}=\int_0^T p^*(t,\cdot)\right|_{\partial \B^*}$$ is the Lagrange multiplier associated with the volume constraint.
\end{enumerate}

Let us now assume that $E=\B^*_{\tau\Phi}$ for a fixed compactly supported vector field $\Phi \in W^{2,p}$ normal to $\partial \B^*$, and for some $\tau\in (0;1)$. We first use a change of variables: let us define

\begin{align*}T_\tau:=Id+\tau\Phi\,, J_{\Sigma,\tau}(\Phi):=\det(\n T_\tau)\left| (^T\n T_\tau^{-1})\nu\right|\,,\\ J_{\O,\tau}:=\det(\n \tau\Phi)\,, A_\tau:=J_{\O,\tau}(\Phi)(Id+\tau\n \Phi)^{-1}(Id+\tau^T\n \Phi)^{-1},\end{align*} $u_\tau:=u_{\B^*_{\tau\Phi}}$ and  $\hat u_\tau:=u_{\B^*_{\tau\Phi}}\circ T_\tau$. By a change of variable, we see that $\hat u_\tau$ satisfies 
\begin{equation}\begin{cases}
J_{\O,\tau}\frac{\partial \hat u_\tau}{\partial \tau}-\nabla \cdot(A_\tau \n \hat u_\tau)=J_{\O,\tau}f^*\text{ in }\OT\,, 
\\ \hat u_\tau(t,\cdot)=0\text{ on }(0;T)\times \partial \O\,, 
\\ \hat u_\tau(0,\cdot)=0,
\end{cases}
\end{equation}
while the function $\hat u'_{\B^*_{\tau\Phi},\Phi}:=u'_{E,\tau\Phi}\circ T_\tau$ which we abbreviate as $\hat u'_\tau$, satisfies

\begin{equation}\label{Sureau}\begin{cases}J_{\O,\tau}\frac{\partial \hat u_\tau'}{\partial t}-\nabla \cdot\left(A_\tau\n \hat u_\tau'\right)=0\text{ in }\OT\,,\\ \llbracket A_\tau \frac{\partial \hat u_\tau'}{\partial \nu}\rrbracket_{\partial \B^*}=-J_{\Sigma,\tau}\left(\Phi\cdot\nu\right)\,, 
\\ \hat u_\tau'=0\text{ on }\partial \O.\end{cases}\end{equation}

With the same notation, $\hat p_\tau\,, \hat p_\tau'$ satisfy
\begin{equation}\begin{cases}
J_{\O,\tau}\frac{\partial \hat p_\tau}{\partial \tau}+\nabla \cdot(A_\tau \n \hat p_\tau)=-J_{\O,\tau}\hat u_\tau\text{ in }\OT\,, 
\\ \hat p_\tau(t,\cdot)=0\text{ on }(0;T)\times \partial \O\,, 
\\ \hat p_\tau(T,\cdot)=0,
\end{cases}
\end{equation}
and
\begin{equation}\label{Sureau}\begin{cases}J_{\O,\tau}\frac{\partial \hat p_\tau'}{\partial t}+\nabla \cdot\left(A_\tau\n \hat p_\tau'\right)=-J_{\O,\tau}\hat u_\tau'\text{ in }\OT\,
\\ \hat p_\tau'=0\text{ on }\partial \O\,, \\ \hat p_\tau'(T,\cdot)=0.\end{cases}\end{equation}
 Finally, we set $\hat{\mathscr H}_\tau:=\mathscr H_{\tau}\circ T_\tau,$ $\mathscr H_\tau$ being the mean curvature of $\B^*_{\tau\Phi}$.

Using these notations, the difference which has to be controlled is hence
\begin{align}
\mathcal L_{\B^*}''(\B^*_{\tau \Phi})[\Phi,\Phi]-\mathcal L_{\B^*}''(\B^*)[\Phi,\Phi]&=\label{R1}\tag{$\bold{R_1}(\tau,\Phi)$} \iint_{(0;T)\times \partial \B^*}\left(\Phi\cdot \nu\right)\left\{J_{\Sigma,\tau}\hat p_\tau'-p_{\B^*}'\right\}
\\&\label{R2}\tag{$\bold{R_2}(\tau,\Phi)$}+\iint_{(0;T)\times \partial \B^*}J_{\Sigma,\tau}\hat{\mathscr H}_\tau\left( \hat p_\tau-p_{\B^*}\right)\left(\Phi\cdot \nu\right)^2
\\&\label{R3}\tag{$\bold{R_3}(\tau,\Phi)$}+\iint_{(0;T)\times \partial \B^*} \left(J_{\Sigma,\tau}\frac{\partial \hat p_\tau}{\partial \nu}-\frac{\partial p_{\B^*}}{\partial \nu}\right)\left(\Phi\cdot \nu\right)^2.
\end{align}

We will control each of these three terms separately and we first recall several geometric estimates from \cite{DambrineLamboley}.

\begin{proposition}[Geometric estimates,{ \cite[Lemma 4.8]{DambrineLamboley}}]\label{Pr:Geom}
For any $p\in (1;+\infty)$, for any $\Phi\in \mathcal X_1(\B^*)\cap W^{2,p}(\O;\R^2)\cap W^{1,\infty}(\O;\R^2)$ such that $\Vert \Phi\Vert_{W^{1,\infty}}\leq M_0$ fixed, there exists a constant $M_p$  such that, for any $\tau\in (0;1)$:
\begin{itemize}
\item \begin{equation}\label{Eq:Jac}\Vert \hat J_{\Sigma,\tau}-1\Vert_{L^\infty(\partial \B^*)}\leq M_p\Vert \Phi\cdot\nu\Vert_{W^{1,\infty}(\partial \B^*)}.\end{equation}
\item \begin{equation}\label{Eq:Curv}\Vert \hat{\mathscr H}_\tau-\mathscr H_{\B^*}\Vert_{L^p(\partial \B^*)}\leq M_p \Vert \Phi \cdot\nu\Vert_{W^{2,p}(\partial \B^*)}.\end{equation}
\item \begin{equation}\Vert A_\tau-Id\Vert_{L^\infty}+\Vert A_\tau-Id\Vert_{W^{1,p}(\O)}\leq M_P \Vert \Phi \Vert_{W^{2,p}( \B^*)}
\end{equation}
\end{itemize}
\end{proposition}
 \paragraph{Control of \eqref{R2}-\eqref{R3}} Our goal is to obtain the existence of a constant $M>0$ and of a modulus of continuity $\eta$ such that 
 \begin{equation}\vert\bold{R_2}(\tau,\Phi)\vert+\vert\bold{R_3}(\tau,\Phi)|\leq M \eta(\Vert \Phi\Vert_{W^{2,p}(\O)})\Vert \Phi\cdot \nu \Vert_{L^2(\partial \B^*)}^2.\end{equation} From Proposition \ref{Pr:Geom} and standard Schauder estimates, such an estimate follows if there exists a modulus of continuity $\eta$ such that, for some $\beta\in (0;1)$, 
 \begin{equation}\label{Eq:Three}
 \Vert u_\tau-u^*\Vert_{\mathscr C^{0,\beta}(\OT)}\leq M \eta(\Vert\Phi\Vert_{W^{2,p}}).
 \end{equation}
 In turn, using Proposition \ref{Pr:Geom} and the H\"{o}lder continuity of $u_\tau$ (Proposition \ref{Pr:Regularity}), \eqref{Eq:Three} is implied by the following: there exist a constant $M>0$ and a modulus of continuity $\eta$ such that 
 \begin{equation}\label{Eq:Barbey}
 \Vert \hat u_\tau-u^*\Vert_{\mathscr C^{0,\beta}(\OT)}\leq M \eta\left(\Vert \Phi\Vert_{W^{2,p}}\right).
 \end{equation}
 \begin{proof}[Proof of \eqref{Eq:Barbey}]
 Straightforward computations show that $z_\tau:=\hat u_\tau-u^*$ solves
 \begin{equation}
 \begin{cases}
J_{\O,\tau} \frac{\partial z_\tau}{\partial \tau}-\nabla \cdot(A_\tau \n z_\tau)=f^*(J_{\O,\tau}-1)+\nabla \cdot((A_\tau-1)\n u^*)+(J_{\O,\tau}-1)\frac{\partial u_{\B^*}}{\partial t}\text{ in }\OT\,, 
 \\ z_\tau(t,\cdot)=0\text{ on }(0;T)\times \partial \O\,, 
 \\ z_\tau(0,\cdot)=0.
 \end{cases}
 \end{equation}
 Standard $L^p$ estimates imply that for any $p\in (1;+\infty)$ there exists a constant $M_p$ such that
 \begin{equation}
 \Vert z_\tau\Vert_{W^{1,p}(\OT)}\leq M_p\left(\Vert \Phi\Vert_{W^{2,p}}+\Vert \Phi\Vert_{W^{1,\infty}}\right)
 \end{equation}
 so that Sobolev embeddings conclude the proof.
 \end{proof}
 
 \paragraph{Control of \eqref{R1}}
 To control this term it suffices to show that there exists a constant $M_p$ such that 
 \begin{equation}\left \Vert \int_0^T \hat p_\tau'-p_{\B^*}'\right\Vert_{L^2(\partial \B^*)}\leq M_p \Vert \Phi\Vert_{W^{2,p}}\Vert \Phi \cdot \nu \Vert_{L^2(\partial \B^*)}.\end{equation} By the continuity of the trace it follows that it suffices to prove that 
 \begin{equation}\label{Eq:Ez3}\int_0^T \Vert \hat p_\tau'-p_{\B^*}'\Vert_{W^{1,2}(\O)}^2\leq M_p \Vert \Phi\Vert_{W^{2,p}}^2\Vert \Phi \cdot \nu \Vert_{L^2(\partial \B^*)}^2.\end{equation}
Let us define $z_\tau':=\hat p_\tau'-p_{\B^*}'$.  Straightforward computations show that
\begin{equation}
\begin{cases}
J_{\O,\tau}\frac{\partial z_\tau'}{\partial t}+\nabla \cdot(A_\tau \n z_\tau')=-\hat u_\tau'+u_{\B^*}'-\n \cdot((A_\tau-1)\n p_{\B^*}')+(J_{\O,\tau}-1)\frac{\partial p_{\B^*}'}{\partial t}\,, 
\\ z_\tau'(T,\cdot)=0\,, 
\\ z_\tau'(t,\cdot)=0\text{ on }(0;T)\times \O
\end{cases}
\end{equation}
and so \eqref{Eq:Ez3} follows from standard $W^{1,2}$ estimates if we can prove that 
\begin{equation}\label{Eq:R}\int_0^T  \left\Vert \frac{\partial p_{\B^*}'}{\partial t}\right\Vert_{L^2(\O)}^2+||\n p_{\B^*}'||_{L^2(\O)}^2\leq M_p\Vert \Phi \cdot \nu \Vert_{L^2(\partial \B^*)}^2\end{equation} and that
\begin{equation}\label{Eq:R2}\int_0^T \Vert u_{\B^*}'-\hat u_\tau'\Vert_{L^2(\O)}^2\leq M_p  \eta(\Vert \Phi\Vert_{W^{2,p}})^2\Vert \Phi \cdot \nu \Vert_{L^2(\partial \B^*)}^2\end{equation} for some constant $M_p$. To prove these two inequalities, we begin with a first estimate: 
\begin{equation}\label{Eq:GS}
\iint_\OT \left(\frac{\partial u_{\B^*}'}{\partial t}\right)^2+\iint_\OT (u_{\B^*}')^2\leq M \Vert \Phi\cdot \nu \Vert_{L^2(\partial \B^*)}^2.\end{equation}
\begin{proof}[Proof of \eqref{Eq:GS}]
For the sake of readability, we abbreviate $u_{\B^*}'$ as $u'$ here. Multiplying the equation on $u'$ by $u'$ and integrating by parts, we obtain 
\begin{equation}
\int_\O (u')^2(T,\cdot)+\iint_\OT |\n u'|^2=\iint_{(0;T)\times \partial \B^*}\left(\Phi\cdot\nu\right)u'\leq \Vert \Phi\cdot \nu\Vert_{L^2(\partial \B^*)}\Vert u'(t,\cdot)\Vert_{L^2((0.T)\times\partial \B^*)}.
\end{equation}
By continuity of the trace and by the Poincar\'e inequality,
 we obtain 
\begin{multline}
\iint_\OT |\n u'|^2\leq \int_\O (u')^2(T,\cdot)+\iint_\OT |\n u'|^2
\leq  M \Vert \Phi \cdot \nu \Vert_{L^2(\partial \B^*)}\Vert \n u'\Vert_{L^2(\OT)}.
\end{multline}
This first gives 

\begin{equation}\Vert \n u'\Vert_{L^2(\OT)}\leq   M \Vert \Phi \cdot \nu \Vert_{L^2(\partial \B^*)}\end{equation} which in turn implies
\begin{equation}\iint_\OT (u')^2\leq M T    \Vert \Phi \cdot \nu \Vert_{L^2(\partial \B^*)}^2.\end{equation} 

To obtain the estimates on $\iint_\OT \left(\frac{\partial u_{\B^*}'}{\partial t}\right)^2$ we proceed as follows: using $\frac{\partial u_{\B^*}'}{\partial t}$ as a test function we obtain 
\begin{align}
\iint_\OT \left(\frac{\partial u_{\B^*}'}{\partial t}\right)^2+\int_\O |\n u_{\B^*}'|^2(T,\cdot)&=\iint_{(0;T)\times \partial \B^*}\left(\Phi\cdot \nu\right)\frac{\partial u_{\B^*}'}{\partial t}
\\&=\int_{\partial \B^*}\left(\Phi\cdot \nu\right)(u_\B')(T,\cdot)
\\&\leq M\Vert \Phi\cdot \nu\Vert _{L^2(\partial \B^*)}\Vert \n u_{\B^*}'\Vert_{L^2(\O)}(T,\cdot),
\end{align}
which gives the conclusion: indeed, we apply Young's inequality $2ab\leq \e a^2+\frac{b^2}\e$ to the right hand side and conclude.

\end{proof}

Let us then turn to \eqref{Eq:R}

\begin{proof}[Proof of \eqref{Eq:R}]
Let us define $q'(t,\cdot):=p_{\B^*}'(T-t,\cdot)$. Then $q'$ satisfies, with $u'=u_{\B^*}'$,
\begin{equation}\begin{cases}
\frac{\partial q'}{\partial t}-\Delta q'=u'(T-t,\cdot)\text{ in }\OT\,, 
\\ q'(0,\cdot)=0\,, 
\\ q'(t,\cdot)=0\text{ on }(0;T)\times \partial \O,
\end{cases}
\end{equation}
so that by standard parabolic estimates for the heat equation we obtain, for some constant $M$
\begin{equation}\iint_\OT \left(\frac{\partial p_{\B^*}'}{\partial t}\right)^2=\iint_\OT\left(\frac{\partial q'}{\partial t}\right)^2\leq MT\iint_\OT (u')^2\leq MT \Vert \Phi\cdot \nu \Vert_{L^2(\partial \B^*)}^2.\end{equation} 

\end{proof}
Finally, let us prove \eqref{Eq:R2}.
\begin{proof}[Proof of \eqref{Eq:R2}]

Let us recall that the weak formulation of the equations on $\hut$ and on $u':=u_{\B^*}'$ are: for any test function $v$, 
\begin{equation}\iint_\OT J_{\O,\tau}\frac{\partial \hut}{\partial t}v+\iint_\OT \langle A_\tau \n \hut,\n v\rangle=\iint_{(0;T)\times \partial \B^*}J_{\Sigma,\tau}\left(\Phi\cdot \nu \right) v
\end{equation}
and 
\begin{equation}
\iint_\OT \frac{\partial u'}{\partial t}v+\iint_\OT \langle \n u',\n v\rangle=\iint_{(0;T)\times \partial \B^*} \left(\Phi\cdot \nu\right)v.
\end{equation}
Substracting these two weak formulations and setting $w_\tau':=\hat u_\tau'-u_{\B^*}'$ we obtain, on $w_\tau'$, the weak formulation 
\begin{multline}
\iint_\OT J_{\O,\tau} \frac{\partial w_\tau'}{\partial t}v+\iint_\OT \langle A_\tau \n w_{\tau}',\n v\rangle=\iint_{(0;T)\times \partial \B^*}(J_{\Sigma,\tau}-1)v+\iint_\OT (J_{\O,\tau}-1)\frac{\partial u'}{\partial t}v
\\+\iint_\OT\langle (A_\tau-Id)\n u',\n v\rangle.
\end{multline}
Using $v=w_\tau'$ as a test function  we obtain in the same way, using Poincar\'e inequality and the continuity of the trace for any $t$, up to a multiplicative constant that does not depend on $\Phi$,
\begin{align}
\int_\O (w_\tau')^2(t,\cdot)+\iint_{(0;t)\times \O} |\n w_\tau'|^2&\leq \Vert J_{\Sigma,\tau}-1\Vert_{L^2(\partial \B)}\Vert \n w_\tau'\Vert_{L^2((0;t)\times \O)}(t)
\\&+\Vert J_{\O,\tau}-1\Vert_{L^\infty}\left\Vert \frac{\partial u'}{\partial t}\right\Vert_{L^2((0;t)\times \O)}\Vert \n w_\tau'\Vert_{L^2((0;t)\times \O)}
\\&+\Vert A_\tau-1\Vert_{L^\infty}\left\Vert \n u'\right\Vert_{L^2((0;t)\times \O)}\left\Vert \n w_\tau'\right\Vert_{L^2((0;t)\times \O)}.
\end{align}
The conclusion then follows.
\end{proof}

\end{proof}

\newpage
\textsc{Idriss Mazari}
\\Technische Universit\"{a}t Wien, Institute of Analysis and Scientific Computing, 8-10 Wiedner Haupstrasse, 1040 Wien (\texttt{idriss.mazari@tuwien.ac.at})

\end{document}